\documentclass[10pt]{article}
\usepackage[francais,english]{babel}
\usepackage{amsmath}
\usepackage{amsfonts}
\usepackage{amssymb}
\usepackage{amsthm}
\usepackage{graphics}
\usepackage{graphicx}
\usepackage{amscd}
\usepackage{fullpage}
\usepackage{epsfig}
\usepackage{color}
\usepackage[matrix,arrow,ps]{xy}

\newcommand{\Z}{\mathbb{Z}}
\newcommand{\R}{\mathbb{R}}
\newcommand{\N}{\mathbb{N}}

\newcommand{\C}{\mathbb{C}}

\newcommand{\U}{\mathbb{U}}
\newcommand{\E}{\mathbb{E}}

\newcommand{\mc}{\mathcal}
\newcommand{\mb}{\mathbb}
\newcommand{\mf}{\mathfrak}
\newcommand{\eps}{\varepsilon}
\newcommand{\ind}{{\bf 1}}
\renewcommand{\P}{\mathbb{P}}

\newcommand{\rK}{{\sf K}}

\newcommand{\uK}{{\underline {\sf K}}}
\newcommand{\rS}{{\sf S}}
\DeclareMathOperator{\dist}{dist}

\renewcommand{\H}{\mathbb{H}}

\newcommand{\Lap}{\Delta\!}

\DeclareMathOperator{\Tr}{Tr}
\DeclareMathOperator{\Id}{Id}

\DeclareMathOperator{\diam}{diam}

\DeclareMathOperator{\SLE}{SLE}
\DeclareMathOperator{\CLE}{CLE}
\DeclareMathOperator{\SL}{SL}
\DeclareMathOperator{\GL}{GL}

\DeclareMathOperator{\Log}{Log}

\title{Double dimers, conformal loop ensembles and isomonodromic deformations}
\newtheorem{thm}{Theorem}

\newtheorem{Thm}[thm]{Theorem}

\newtheorem{Lem}[thm]{Lemma}
\newtheorem{Cor}[thm]{Corollary}

\begin{document}
\maketitle
\begin{abstract}
The double-dimer model consists in superimposing two independent, identically distributed perfect matchings on a planar graph, which produces an ensemble of non-intersecting loops. In \cite{Ken_DD}, Kenyon established conformal invariance in the small mesh limit by considering topological observables of the model parameterized by $\SL_2(\C)$ representations of the fundamental group of the punctured domain. The scaling limit is conjectured to be $\CLE_4$, the Conformal Loop Ensemble at $\kappa=4$ \cite{SheWer_CLE}.

In support of this conjecture, we prove that a large subclass of these topological correlators converge to their putative $\CLE_4$ limit. Both the small mesh limit of the double-dimer correlators and the corresponding $\CLE_4$ correlators are identified in terms of the $\tau$-functions introduced by Jimbo, Miwa and Ueno \cite{JMU} in the context of isomonodromic deformations.
\end{abstract}

\section{Introduction}

The dimer (or perfect matching) model is a classical model of Statistical Mechanics. On a fixed graph $\Gamma=(V,E)$, a dimer configuration consists in a subset ${\mf m}$ of edges such that each vertex is the endpoint of exactly one edge in ${\mf m}$. For planar graphs, Kasteleyn's fundamental result \cite{Kas_dimerstat} states that the partition function of this model can be evaluated as the Pfaffian of a signed adjacency matrix, the Kasteleyn matrix.

The high degree of solvability of the dimer model makes it amenable to asymptotic analysis for large planar graphs (or graphs with fine mesh). In celebrated work, Kenyon \cite{Ken_domino_conformal,Ken_domino_GFF} established conformal invariance and convergence to the Gaussian free field of the height function representation of the dimer configuration. 

In parallel, Schramm \cite{Sch99} introduced a way to describe conformally invariant random interfaces: Schramm-Loewner Evolutions ($\SLE$). This was later extended to describe collections of non-intersecting loops, resulting in Conformal Loop Ensembles ($\CLE$, see in particular in \cite{SheWer_CLE}). These proved a powerful tool to describe and analyze the scaling limit of some discrete models. In the cases of the Loop-Erased Random Walk and Uniform Spanning Tree \cite{LSW_LERW}, of critical percolation \cite{Smi_perc}, the Ising model \cite{Smi_Ising}, and the discrete Gaussian Free Field \cite{SS_freefield}, the convergence to $\SLE$ is established by the martingale observable method. This consists in evaluating the fine mesh limit of a single observable for a collection of domains (the subdomains obtained by exploration along a critical interface).

As observed by Percus \cite{Per_dimer}, superimposing two dimer configurations on a given graph results in a collection of non-intersecting simple loops (and doubled edges), see Figure \ref{Fig:DD}.
\begin{figure}
\begin{center}\includegraphics[scale=.8]{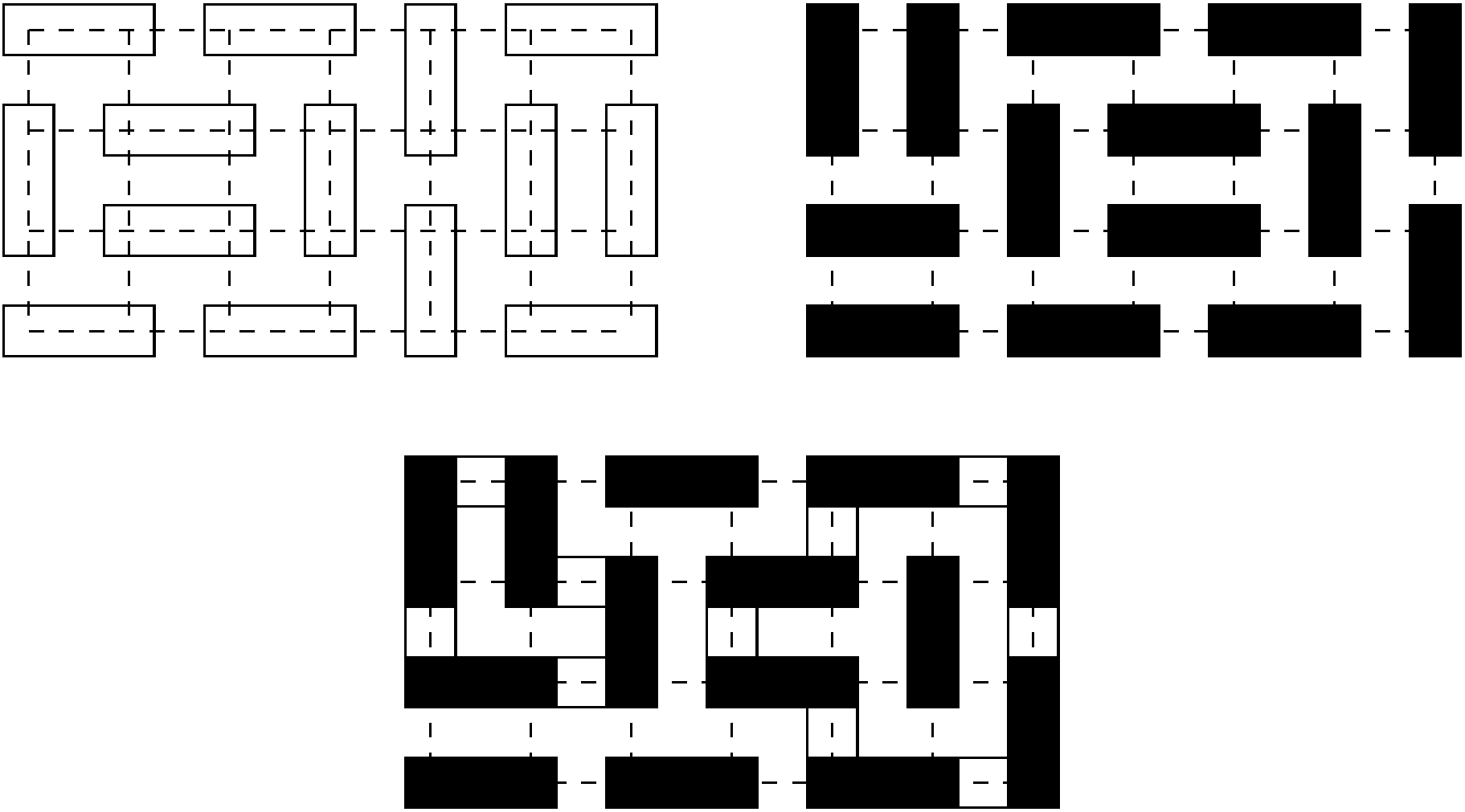}\end{center}
\caption{Top-left: dimer configuration (rectangles) on a portion of the square-lattice (dashed). Top-right: another such configuration. Bottom: resulting double-dimer configuration.}\label{Fig:DD}
\end{figure}
The {\em double-dimer model} consists in sampling two independent dimer configurations and considering the resulting loop ensemble. For bipartite graphs, these may also be seen as level lines of a height function, known \cite{Ken_domino_GFF} to converge to the GFF along suitable sequences of graphs. In agreement with computations in \cite{Ken_domino_conformal,Ken_domino_GFF,Sch_percform}, this led Kenyon to conjecture convergence of chordal double-dimer paths and full double-dimer configurations to $\SLE_4$ and $\CLE_4$ respectively. This is also consistent with the results on Discrete Gaussian Free Field level lines established in \cite{SS_freefield}.

The main obstruction to implement the martingale observable method in the case of the double-dimer is that the dimer model is extremely sensitive to boundary conditions. In particular, the ``sublinear height" type of boundary conditions which are amenable to asymptotic analysis are not dynamically stable under exploration of the domain along double-dimer paths. This is somewhat analogous to the main technical obstruction encountered in the analysis of level lines of the DGFF \cite{SS_freefield}.

In a sharp departure from the martingale observable method, Kenyon \cite{Ken_DD} introduced the idea of analyzing a collection of ``non-commutative" observables in a single domain (rather than a single observable in a collection of domains). Building on Kasteleyn's determinantal solvability and Mehta's theory of quaternionic determinants, he constructed double-dimer observables (in a domain $D$, with punctures - or macroscopic holes - $\lambda_1,\dots,\lambda_n$) parameterized by a representation $\rho:\pi_1(D\setminus\{\lambda_1,\dots,\lambda_n\})\rightarrow\SL_2(\C)$ whose expected value
$$\E_{{\rm dimer}^2}\left(\prod_{\ell{\rm\ loop\ in\ }{\mf m}_1\cup{\mf m}_2}\frac{\Tr(\rho(\ell))}2\right)$$
can be evaluated as the determinant of an operator derived from the Kasteleyn matrix. This operator can be thought as a Kasteleyn operator on a flat rank 2 bundle with monodromy given by $\rho$. This allowed him to establish conformal invariance of the double-dimer model (in the sense of random laminations \cite{Ken_DD}). The asymptotic analysis of the dimer model (in the conformally invariant regime) is based on the interpretation of the Kasteleyn operator as a discrete Cauchy-Riemann operator. This suggests interpreting the (scaling limit of) Kenyon's correlators as regularized determinants of Cauchy-Riemann operators on flat bundles. 

The question of monodromy representations of differential equations (of a complex variable, with poles at prescribed punctures) and the resulting isomonodromic deformation problem is a classical one, featured in Hilbert's twenty-first problem (see eg \cite{BdM_MathPhys,FIKN_Painleve}). An isomonodromic deformation preserves the monodromy representation under (local) displacement of the punctures. In the context of the Ising model, Jimbo, Miwa and Uneo \cite{JMU} introduced a numerical invariant, the $\tau$-function, attached to an isomonodromic deformation. Palmer \cite{Pal_tau} proposed to interpret the $\tau$-function as a regularized determinant of an operator on sections of a bundle over the punctured sphere.

In the present work our goal is to establish that, in the small mesh limit, a large collection of Kenyon's correlators converge to their value for $\CLE_4$, in strong support of the conjectured convergence of double-dimer loops to $\CLE_4$. The connection between double dimers and $\CLE_4$ is established via $\tau$-functions. Our main result concerns double-dimers in $\delta(\Z\times\N)$, the upper half-plane square lattice with small mesh $\delta$. 

\begin{Thm}
Fix punctures $\lambda_1,\dots,\lambda_n\in\H$. Then if $\rho:\pi_1(\H\setminus\{\lambda_1,\dots,\lambda_n\})\rightarrow\SL_2(\R)$ is a representation with unipotent local monodromies close enough to the identity, we have
\begin{align*}
\lim_{\delta\searrow 0}\E^{\delta(\Z\times\N)}_{{\rm dimer}^2}\left(\prod_{\ell{\rm\ loop\ in\ }{\mf m}_1\cup{\mf m}_2}\frac{\Tr(\rho(\ell))}2\right)=
\tau(\H\setminus\{\lambda_1,\dots,\lambda_n\};\rho)=\E_{\CLE_4(\H)}\left(\prod_{\ell\in\CLE_4(\H)}\frac{\Tr(\rho(\ell))}2\right)
\end{align*}
\end{Thm}
The definition of the LHS is discussed in Section \ref{Sec:DDHP}; and that of the $\tau$-function in Sections \ref{Sec:tau},\ref{Sec:tauHP}
\begin{proof}
From Lemma \ref{Lem:conv} and Theorem \ref{Thm:CLEtau}.
\end{proof}

This has consequences for more direct aspects of the scaling limit.

\begin{Cor}
Let $N^\delta_{xy}$ be the number of double-dimer loops on $\delta(\Z\times\N)$ encircling $x\neq y$; and $N_{xy}$ be the number of loops around $x,y$ in a (nested) $\CLE_4$. Then, as $\delta\searrow 0$, $N_{xy}^\delta$ converges in law to $N_{xy}$  
\end{Cor}
\begin{proof}
See Section 10 in \cite{Ken_DD}.
\end{proof}

In establishing conformal invariance of double-dimers in \cite{Ken_DD} (as laminations), Kenyon relies on results from \cite{FocGon} enabling him to characterize probability measures (on finite laminations in multiply-connected domains) by the values of $\SL_2(\C)$ observables. One issue is the accumulation of small loops nested around each puncture, which is circumvented in the present work by restricting to the case of unipotent local monodromies. By general principles (see Section \ref{Sec:SLECLE} for a discussion of loop ensembles), we have

\begin{Cor}
Let $\mu_\delta$ be the measure on loop ensembles induced by double dimers on $\delta(\Z\times\N)$. The assumptions:
\begin{enumerate}
\item $(\mu_\delta)_\delta$ is tight, and
\item  a probability measure $\mu$ on loop ensembles in a simply-connected domain $D$ is uniquely characterized by the expectations of the functionals
$$\int\prod_{\alpha}\frac{\Tr(\rho(\ell_\alpha))}2d\mu((\ell_\alpha)_\alpha)$$
where for any set of punctures $\{\lambda_1,\dots,\lambda_n\}$, $\rho:\pi_1(D\setminus\{\lambda_1,\dots,\lambda_n\})\rightarrow\SL_2(\R)$ is a representation with unipotent local monodromies close enough to the trivial one,
\end{enumerate}
imply weak convergence of the $\mu_\delta$'s to the $\CLE_4(D)$ measure as $\delta\searrow 0$.
\end{Cor}

We chose to phrase in the simplest framework, viz. double-dimers on $\Z\times\N$. Given known arguments, these results can be extended in several directions (at some technical and notational cost):
\begin{enumerate}
\item The domain can be chosen as a general simply-connected domain, with ``Temperleyan" boundary conditions of the type considered in \cite{Ken_domino_conformal} (or the Dirichlet/Neumann conditions natural in the context of Peano exploration of the Uniform Spanning Tree, see \cite{LSW_LERW}). For the near-boundary estimates \ref{Lem:nearbounddiscr}, it is convenient to have a small but macroscopic flat segment on the boundary.
\item Instead of the square lattice, the graph can be modelled on a Temperleyan isoradial graph, as in \cite{Ken_critical}.
\item Instead of $\SL_2(\R)$-monodromies, one can use general (locally unipotent) $\SL_2(\C)$-valued monodromies. This does not create genuine additional difficulties but muddles a bit the discussion.
\end{enumerate}

The article is organized as follows. Section 2 provides background material on dimers, $\tau$-functions, and conformal loop ensembles. Section 3 establishes convergence of double-dimer observables to $\tau$-functions. Section 4 evaluates the corresponding observables for the conformal loop ensemble $\CLE_4$, also in terms of $\tau$-functions. Miscellaneous technical estimates are relegated to Section 5.

\section{Background and notations}

\subsection{Dimers and double dimers}

Consider a graph $\Gamma=(V,E)$; a {\em dimer configuration} (or {\em perfect matching}) consists in a subset of edges (dimers) of the graph such that each vertex is the endpoint of exactly one edge in the configuration. For a finite weighted graph ($\omega(e)$ denotes the weight of the edge $e$), the weight of a dimer configuration ${\mf m}$ is
$$w({\mf m})=\prod_{e\in{\mf m}}\omega(e)$$
For nonnegative weights, one can consider the probability measure $\P$ on dimer configurations such that $\P\{{\mf m}\}\propto w({\mf m})$: this is the {\em dimer model}. Let us now specialize to oriented bipartite graphs. Vertices are thus partitioned into black ($B$) and white ($W$) vertices and we assume that $|B|=|W|$ (otherwise there is no dimer configuration). The {\em Kasteleyn-Percus matrix} $\rK:\C^B\rightarrow\C^W$ is (a block of the) signed, weighted adjacency matrix with matrix elements: $\rK(w,b)=\pm \omega(bw)$, depending on whether the edge $(bw)$ is oriented from $b$ to $w$ or from $w$ to $b$.\\
Kasteleyn's fundamental enumeration result \cite{Kas_dimerstat} states that for any given (finite) planar graph, one can construct and characterize an orientation s.t.:
$${\mc Z}_{\rm dimer}=\sum_{{\mf m}}w({\mf m})=\pm\det(\rK)$$
The determinant is evaluated w.r.t. ``canonical" bases on $\C^B$, $\C^W$, which are given up to permutation, hence the sign ambiguity. This is an algebraic identity in the variables $\omega(e)$, $e\in E$. Consequently, if we consider a modified Kasteleyn matrix $\rK_\rho$ with matrix elements $\rK_\rho(w,b)=\rK(w,b)\rho(bw)$, we have:
$$\E\left(\prod_{e\in{\mf m}}\rho(bw)\right)=\frac{\det(\rK_\rho)}{\det(\rK)}$$
(where we restrict to the bipartite case for simplicity). From now on we are only considering planar bipartite graphs.\\
For instance, one may  consider several disjoint ``branch cuts", viz. disjoint simple paths $\gamma_i$ from $f_i$ to $f'_i$ on the dual graph, $i=1,\dots,n$; $\chi_1,\dots,\chi_n\in\U$ are given roots of unity. Set $\rho(bw)=\chi_i$ (resp. $\chi_i^{-1}$) if $\gamma_i$ crosses $(bw)$ with $b$ (resp. $w$) on its lefthand side, and $\rho(bw)=1$ otherwise. Then it is easy to check that $\prod_{e\in{\mf m}}\rho(bw)$ does not depend on the position of the branch cuts, can be expressed simply in terms of the height function. The asymptotics of $\E\left(\prod_{e\in{\mf m}}\rho(bw)\right)$ in the large scale regime are studied in some details in \cite{Dub_tors}, where they are compared with the corresponding free field electric correlators. The data $\rho$ corresponds to a unitary character of $\pi_1(\Sigma)$, where $\Sigma$ is the punctured plane $\C\setminus\{f_1,\dots,f'_n\}$; the analysis is based on comparing $(\rK^\rho)^{-1}$ to the Cauchy kernel for the unitary line bundle over $\Sigma$ specified by $\rho$.\\

In the double-dimer problem, one samples independently two dimer configurations ${\mf m}_1,{\mf m}_2$ on the same graph according to the same (weighted) dimer measure (it is also possible to consider two distinct systems of edge weights). The superposition ${\mf m}_1\cup{\mf m}_2$ of the two matchings produces a collection of simple closed loops and doubled edges which covers the (vertices of) the underlying graph. Trivially,
$${\mc Z}_{{\rm dimer}^2}=\sum_{{\mf m}_1,{\mf m}_2}w({\mf m}_1)w({\mf m}_2)=\det(\rK\oplus\rK)$$
where $\rK\oplus\rK:\C^B\oplus\C^B\simeq (\C^2)^B\longrightarrow\C^W\oplus\C^W\simeq (\C^2)^W$. Consider $f_0,\dots,f_n$ (centers of) faces of the underlying graph; $\Sigma=\hat\C\setminus\{f_0,\dots,f_n\}$ a punctured sphere; $\rho:\pi_1(\Sigma)\rightarrow \SL_2(\C)$ a representation of the fundamental group of the punctured sphere.  For concreteness, one may choose a sequence of disjoint cuts $\gamma_1,\dots,\gamma_n$ (simple paths on the dual graph) with $\gamma_i$ going from $f_{i-1}$ to $f_i$ and a matrix $\chi_i\in \SL_2(\C)$. If $\ell_i$ is a counterclockwsie (ccwise) oriented loop with $f_i$ in its interior and other punctures in its exterior, set $\rho(\ell_i)=\chi_i^{-1}\chi_{i+1}$ (with $\chi_{-1}=\chi_{n+1}=\Id_2$). Since $\pi_1(\Sigma)$ is a group with generators $\ell_0,\dots,\ell_n$ and relations $\gamma_0\dots\gamma_n=\Id$, the data of $\rho$ and $(\chi_1,\dots,\chi_n)$ are equivalent. This data also defines a rank 2 flat holomorphic bundle over $\Sigma$.

There is a natural way to twist $\rK\oplus \rK$ by the character $\rho$. If $(bw)$ crosses $\gamma_i$, the $2\times 2$ block $K(w,b)\Id_2$ corresponding to $(wb)$ in $\rK\oplus \rK$ is replaced with the block $\rK(w,b)\chi_i$ (resp. $\rK(w,b)\chi_i^{-1}$) if $b$ (resp. $w$) is on the LHS of $\gamma_i$; leaving all other matrix elements unchanged, this defines a twisted operator $(\rK\oplus\rK)_\rho$. Based on Mehta's notion of quaternionic determinants \cite{Meh_RM}, Kenyon \cite{Ken_DD} established the following remarkable identity:
\begin{equation}\label{eq:KenDD}
\E_{{\rm dimer}^2}\left(\prod_{\ell{\rm\ loop\ in\ }{\mf m}_1\cup{\mf m}_2}\frac{\Tr(\rho(\ell))}2\right)=\frac{\det((\rK\oplus\rK)_\rho)}{\det(\rK\oplus\rK)}
\end{equation}
Remark that $\Tr(\rho(\ell))$ depends neither on the choice of a base point for the fundamental group nor on the orientation of $\ell$ (since $\rho(\ell)$ is in $\SL_2(\C)$ and trivially $\left(\begin{array}{cc}
a&b\\
c&d
\end{array}
\right)^{-1}=\left(\begin{array}{cc}
d&-b\\
-c&a
\end{array}
\right)$ if $ad-bc=1$).

\subsection{Isomonodromic deformations and $\tau$-functions}\label{Sec:tau}

Here we briefly sketch some elements of the classical theory of isomonodromic deformations; for comprehensive surveys, see eg \cite{BdM_MathPhys,Yosh,FIKN_Painleve}.

Let $\lambda_1,\dots,\lambda_n$ be given punctures on the Riemann sphere $\hat \C=\C{\mb P}^1$ ($\lambda_i\neq\infty$), $\Sigma=\hat\C\setminus\{\lambda_1,\dots,\lambda_n\}$. Let $A_1,\dots,A_n\in M_r(\C)$ be given and consider the Fuchsian equation (vector- or matrix-valued, holomorphic, linear ODE with regular singular points)
\begin{equation}\label{eq:Fuchs}
\partial_z Y(z)=A(z)Y(z)=\sum_{i=1}^n\frac{A_i}{z-\lambda_i}Y(z)
\end{equation}
The $\lambda_i$'s are regular singular points in the sense that solutions have polynomial growth near these singularities: $Y(z)=O((z-\lambda)^{-N})$ for some $N$.

Furthermore we assume that $\sum_i A_i=0$ so that $\infty$ is a regular point. Local solutions define a rank $r$ local system (a sheaf in dimensional vector spaces over $\Sigma$), by Cauchy-Kowalevski. If $U$ is a simply connected neighborhood of $\infty$ in $\Sigma$, let $Y_0$ be the unique solution on $U$ with $Y_0(\infty)=\Id_r$. Then $\partial_z\log\det(Y_0(z))=\Tr(\partial_z Y_0(z)Y_0^{-1}(z))=\sum_i\Tr(A_i)/(z-\lambda_i)$. Hence if the $A_i$'s are in ${\mf {sl}}_r(\C)$, $Y_0$ takes values in $\SL_r(\C)$ (up to a gauge change, one may assume this is the case).

Given any loop $\gamma$ in $\Sigma$ rooted at infinity, one can compute the solution $Y_0$ along $\gamma$ (starting at $Y_0(\infty)=\Id$ and analytically continuing it along $\gamma$); plainly the resulting matrix depends only on the homotopy class of $\gamma$ and this defines a monodromy representation: $\rho:\pi_1(\Sigma)\rightarrow \GL_r(\C)$. More precisely, the analytic continuation of a germ $Y_0$ at infinity along the loop $\gamma$ is a germ $Y_0\rho(\gamma)^{-1}$. Replacing $A_i$ with $GA_iG^{-1}$ ($G\in \GL_r(\C)$ fixed) leads to the conjugate representation $G\rho G^{-1}$. 

This procedure associates to a Fuchsian system \eqref{eq:Fuchs} (determined by the punctures $\lambda_i$ and matrices $A_i$) a monodromy representation $\rho$. The fundamental group $\pi_1(\Sigma)$ has a set of generators $\gamma_1,\dots,\gamma_n$ s.t. $\gamma_i$ encircles $\lambda_i$ counterclockwise and no other puncture; they satisfy the relation $\gamma_1\dots\gamma_n=\Id$. The data of a representation $\pi_1(\Sigma)\rightarrow\GL_r(\C)$ is thus equivalent to the data of monodromy matrices $M_1=\rho(\gamma_1),\dots,M_n=\rho(\gamma_n)$ with $M_1\dots M_n=\Id_r$. Hilbert's twenty-first problem consists in finding $A_i$'s producing a prescribed representation.

We consider now displacements of punctures $\lambda_1,\dots,\lambda_n$. Let $A_i=A_i(\lambda_1,\dots,\lambda_n)$ (maintaining $\sum_i A_i=0$). If $(\lambda'_1,\dots,\lambda'_n)$ is close to $(\lambda_1,\dots,\lambda_n)$, there is a natural identification $\pi_1(\Sigma')\simeq\pi_1(\Sigma)$, where $\Sigma'=\hat\C\setminus\{\lambda'_1,\dots,\lambda'_n\}$. An {\em isomonodromic deformation} of \eqref{eq:Fuchs} is s.t. the monodromy representation is locally constant. The {\em Schlesinger equations}
\begin{equation}\label{eq:Schlesinger}
\partial_{\lambda_j}A_i=\frac{[A_i,A_j]}{\lambda_i-\lambda_j}
\end{equation}
for $j\neq i$ ensure that the monodromy representation $\rho=\rho(\lambda_1,\dots,\lambda_n)$ is indeed locally constant. (The last derivative $\partial_iA_i$ is fixed by the condition $\sum_j A_j=0$). Alternatively, one may write 
$$dA_i=-\sum_{j:j\neq i}[A_i,A_j]\frac{d(\lambda_i-\lambda_j)}{\lambda_i-\lambda_j}$$
Under \eqref{eq:Schlesinger}, one can extend \eqref{eq:Fuchs} in order to account simultaneously for the dependence on the punctures $\lambda=(\lambda_1,\dots,\lambda_n)$ and the ``spectator point" $z$ for the fundamental solution $Y_0$ normalized by $Y_0(\infty)=\Id$:
\begin{equation}\label{eq:Fuchstot}
dY_0(z;\lambda)=\sum_j A_jY_0\frac{d(z-\lambda_j)}{(z-\lambda_j)}
\end{equation}
Writing this equation under the general Pfaffian form $dY=\Omega Y$ ($\Omega$ a matrix-valued 1-form in the variables $z,\lambda_1,\dots,\lambda_n$), the integrability condition $d\Omega+\Omega\wedge\Omega=0$ boils down to \eqref{eq:Schlesinger}. In other terms, the nonlinear Schlesinger equations \eqref{eq:Schlesinger} are the integrability conditions of the linear equations \eqref{eq:Fuchstot}. 

Following Jimbo-Miwa-Ueno \cite{JMU} (see also \cite{BdM_MathPhys}), set
\begin{equation}\label{eq:dlogtau}
\omega=\frac 12\sum_{(i,j):i\neq j}\Tr(A_iA_j)\frac{d(\lambda_i-\lambda_j)}{\lambda_i-\lambda_j}
\end{equation}
a 1-form in $\lambda_1,\dots,\lambda_n$. It follows from the Schlesinger equations \eqref{eq:Schlesinger} that $\omega$ is closed: $d\omega=0$. Consequently, setting $d\log\tau=\omega$ defines (at least locally and up to multiplicative constant) a {\em $\tau$-function} on $\hat \C^n\setminus\Delta$, where $\Delta=\{(\lambda_1,\dots,\lambda_n):\exists i<j{\rm\ s.t.\ }\lambda_i=\lambda_j\}$.

Consider the rank $r$ flat bundle $V$ corresponding to $\rho$. One may identify its sections with functions $f:\tilde\Sigma\rightarrow(\C^r)^t$ ($\tilde\Sigma$ the universal cover, values are row vectors) satisfying: $s(\gamma^*x)=s(x)\rho(\gamma)$ (where $\pi_1(\Sigma)$ operates on fibers of $\tilde\Sigma$, $(\gamma_1\gamma_2)^*x=\gamma_2^*\gamma_1^*x$). If $Y_0$ is a fundamental solution, then sections may be written $s=fY_0$ with $f:\Sigma\rightarrow(\C^r)^t$.

Let us now explain some connections to Riemann-Hilbert problems. Consider a branch cut $\gamma$ running from, say, $\lambda_1$ to $\lambda_n$ and going through each puncture $\lambda_1,\dots,\lambda_n$. Then the fundamental solution $Y_0$ is single-valued on $\hat\C\setminus\gamma$. Let $\lambda$ be a point on $\gamma$ (oriented from $\lambda_1$ to $\lambda_n$); $\lambda_-$ (resp. $\lambda_+$) denotes $z$ approached from the left (resp. right) of $\gamma$. Let $\ell$ be a counterclockwise loop rooted at infinity and crossing the cut $\gamma$ once at $\lambda$. One verifies easily that $Y_0(\lambda_+)^{-1}Y_0(\lambda_-)=\rho(\ell)^{-1}$. If moreover the $A_i$'s are traceless and with eigenvalues of modulus $<\frac 12$, we have:
\begin{align*}
Y_0\in\SL_2(\C), Y_0(\infty)=\Id_2\\
Y_0(\lambda_-)=Y_0(\lambda_+)N_i&&\forall \lambda\in(\lambda_i,\lambda_{i+1})\subset\gamma\\
Y_0(z)=o(|z-\lambda_i|^{-1/2})&&\forall i
\end{align*}
If $\tilde Y_0$ is another such matrix-valued analytic function on $\hat\C\setminus\gamma$ with the same prescribed jumps $N_i$ and growth conditions near the punctures, we see that $Y_0\tilde Y_0^{-1}$ is single-valued with removable singularities and is thus constant. This gives a unique characterization of $Y_0$ as a solution of a Riemann-Hilbert problem.

Let us now consider $S(.,w)$ a $\GL_2(\C)$-valued meromorphic function a simple pole with residue $\Id_2$ at $w$ (and no other pole on $\hat\C\setminus\gamma$) and the same jump and growth conditions (near the punctures) as $Y_0$; and $S(\infty,w)=0$. By considering $S(.,w)Y_0^{-1}$ we see that $S$ is uniquely characterized and given by
$$S(z,w)=\frac{Y_0(w)^{-1}Y_0(z)}{z-w}$$
Near $w$ we have
$$S(z,w)=\frac{\Id_r}{z-w}+R(w)+O(z-w)$$
with the {\em Robin kernel} 
$$R(w)=Y_0^{-1}(w)\partial_w Y_0(w)=Y_0^{-1}(w)(\sum_i\frac {A_i}{w-\lambda_i})Y_0(w)$$ 
Let 
$$R_i=R_i(\lambda)=\lim_{w\rightarrow\lambda_i}(Y_0(w)R(w)Y_0^{-1}(w)-\frac{A_i}{w-\lambda_i})=\sum_{j\neq i}\frac{A_j}{\lambda_j-\lambda_i}$$
Then
$$\partial_{\lambda_i}\log\tau=\Tr(A_iR_i)$$
ie the $\tau$-function may be recovered from the kernel $S$ and in particular its behavior near the punctures. This gives an expression for the $\tau$-function in terms of a RH problem.

\subsection{SLE and CLE}\label{Sec:SLECLE}

In this section, we briefly discuss Schramm-Loewner Evolutions (SLEs) and the related Conformal Loop Ensembles (CLEs). See in particular \cite{Wer_Flour,Law_AMS,SheWer_CLE} for general background and complements.

If $Z=\sqrt \kappa B$, $B$ a standard Brownian motion, $\kappa>0$, the Loewner equation reads:
$$\frac{d}{dt}g_t(u)=\frac{2}{g_t(u)-Z_t}$$
with $g_0(u)=u$ for $u$ in the upper half-plane $\H=\{u:\Im(u)>0\}$. Let $\tau_u$ be the time of explosion of this ODE, $K_t=\overline\{u:\tau_u\leq t\}$ the ``hull". Then $g_t:\H\setminus K_t\rightarrow \H$ is a conformal equivalence (characterized by $g_t(u)=z+\frac{2t}u+o(u^{-1})$ for $z$ large). The random Loewner chain $(g_t)){t\geq 0}$ defines the Schramm-Loewner Evolution $\SLE_\kappa$.

It is known that $\gamma_t=\lim_{\eps\searrow 0}g_t^{-1}(i\eps)$ is a.s. a continuous path in $\overline\H$, the {\em trace} of the $\SLE$, s.t. $\H\setminus K_t$ is the unbounded connected component of $\H\setminus\gamma_{[0,t]}$. For $\kappa\leq 4$, it is a.s. a simple path in $\overline\H$ (with $\gamma_0=W_0$, $\gamma_t\in\H$ for $t>0$, and $\lim_t\gamma_t=\infty$). If $\gamma$ is an $\SLE_\kappa$ in $\H$ (started at $0$, aiming at $\infty$), $D$ is simply-connected domain with marked boundary points $x,y$, and $\phi:(\H,0,\infty)\rightarrow (D,x,y)$ is a conformal equivalence, then $\phi(\gamma)$ is (by definition and up to time change) a chordal $\SLE_\kappa$ in $(D,x,y)$ (conformal invariance)

From the Markov property of Brownian motion, we see that the law of a chordal $\SLE_\kappa$ in $(\H,0,\infty)$ conditionally on $K_t$ is that of a chordal $\SLE_\kappa$ in $(\H\setminus K_t,\gamma_t,\infty)$ (Domain Markov property). The chordal $\SLE$ measures are characterized by the conformal invariance and domain Markov properties.

A useful variant of $\SLE_\kappa$ is given by the so-called $\SLE_\kappa(\rho)$'s ($\kappa>0,\rho\in\R$). We distinguish on the real line a seed $Z_0$ and a force point $W_0$ and consider the SDE
$$\left\{\begin{array}{ll}
dZ_t&=\sqrt\kappa dB_t+\frac{\rho}{Z_t-W_t}dt\\
dW_t&=\frac{2}{W_t-Z_t}dt
\end{array}\right.$$
with $g_t$ defined from the driving process $Z$ as above, so that $W_t=g_t(W_0)$. Existence of solutions (in particular extending these solutions past $\tau_{W_0}$) depends on $\kappa,\rho$ and will be discussed in some details in Section \ref{Sec:kapparho}.

From the point of view of statistical mechanics, it is natural to consider loop ensembles. We start from a simply connected $D$, which (up to conformal equivalence and for simplicity) we can take to be the unit disk. The loop-space ${\mc L}_D$ consists of $1-1$ mappings $\U\setminus \overline D$ ($\U$ the unit circle), up to composition by a homeomorphism of $\U$. This is metrized by
$$d(\gamma,\gamma')=\inf_{\sigma:\U\rightarrow\U}\|\gamma-\gamma'\circ\sigma\|_\infty$$
where $\sigma$ runs over homeomorphisms $\U\rightarrow\U$. Without the injectivity condition, this defines a Polish space; so ${\mc L}_D$ is a Lusin space. In a variant, one considers oriented loops (given up to reparameterization by a direct homeomorphism of $\U$).

A loop ensemble in $D$ is an ${\mc L}_D$-valued point process (supported on mutually disjoint loops) which is locally finite in the sense that, for any $\eps>0$, there are finitely loops of diameter $\geq\eps$ (in particular, there are countably many loops). The configuration can be metrized by declaring the configurations $(\gamma_i)_{i\in I}$, $(\gamma'_j)_{j\in J}$ are at distance $\leq\eps$ if they can be paired up to loops of diameter $\leq\eps$ so that paired loops are at distance $\leq\eps$ in ${\mc L}_D$. Classical tightness criteria (based on annular crossing estimates) are given in \cite{AizBur,AizBurNewWil}.

One may also more combinatorial encodings of the loop ensemble. Let $S$ be a countable dense subset of $D$; for $x_1,\dots,x_n\in S$, $N(x_1,\dots,x_n)$ denotes the number of loops in the loop ensemble with $x_1,\dots,x_n$ in the closure of their interior. For instance, if $x,y\in S$, on the event $\{N_{x,y}\geq 1\}$, the innermost loop disconnecting $x,y$ from $\partial D$ is the boundary of $\{z\in S: N_{x,y,z}=0\}$. So a 
(simple) loop ensemble can be reconstructed from the data $(N_T)_{T\subset S,|T|<\infty}$.

Much in the same way that $\SLE$ describes conformally invariant distributions on chordal curves, there is a 1-parameter family $\CLE_\kappa$ (the Conformal Loop Ensembles) of measures on loop ensembles, characterized by a suitable Domain Markov property \cite{SheWer_CLE}. We consider here nested $\CLE$s (so that each loop contains an identically distributed - up to conformal equivalence - $\CLE$, independent of the configuration outside of the loop) consisting of simple disjoint loops, which is the case when $\kappa\in(\frac 83,4]$. The two main approaches to $\CLE$ \cite{SheWer_CLE} are via branching exploration \cite{Sheff_explor} and loop soup percolation \cite{Wer_CRASloop}. We shall need the first one for computational purposes and the second one for a priori estimates (Lemma \ref{Lem:CLEarm}).

In the branching exploration description, one runs a $\SLE_{\kappa}(\kappa-6)$ started, say, from $(0,0^+)$ (see Section \ref{Sec:kapparho} for more details). This creates loops (corresponding to excursion of the real process $W-Z$); then the interior of these loops is explored recursively.

\subsection{$\tau$-functions in simply connected domains}\label{Sec:tauHP}

Here we will need to consider not the classical case of the punctured sphere, but a punctured simply-connected domain (which we can take to be the upper half-plane $\H=\{z\in\C:\Im z>0\}$); this will build on the spherical case and reflection arguments. In the situation of interest to us, we will also need to fix the multiplicative constant in front of the $\tau$-functions.

Consider the upper half-plane $\H$ with punctures $\lambda_1,\dots,\lambda_n\in\H$. Let $\gamma_i$ be a simple counterclockwise loop separating $z_i$ from the other punctures and rooted at $\infty$. We choose the $\gamma_i$'s to be disjoint (except at $\infty$) and in counterclockwise order seen from $\infty$. Classically, the fundamental group $\pi_1(\H\setminus\{\lambda_1,\dots,\lambda_n\})$ is the free group on $n$ generators $\gamma_1,\dots,\gamma_n$. Consequently, the data of a representation $\pi_1\rightarrow\SL_2(\R)$ is equivalent to the data of $n$ matrices $\rho(\gamma_1),\dots,\rho(\gamma_n)$.

Throughout we will choose these matrices to be $2\times 2$, real and unipotent.

We proceed by reflection and introduce punctures $\lambda_{\bar 1},\dots,\lambda_{\bar n}$ in the lower half-plane $-\H$, with the short hand notation
$$\bar k=2n+1-k$$
for $k=1,\dots,2n$ (this depends implicitly on the total number of punctures). We shall be in particular interested in the case where $\lambda_{\bar k}=\overline{\lambda_k}$ for all $k$.

We consider simple clockwise loops $\gamma_{\bar k}$, $k=1,\dots,n$ s.t. $\gamma_{\bar k}$ separates $z_{\bar k}$ from the other punctures, and $\gamma_1,\dots,\gamma_n,\gamma_{\bar n},\gamma_{\bar 1}$ are disjoint except at $\infty$ and listed in counterclockwise order seen from $\infty$. See Figure \ref{Fig:cut}.
\begin{figure}
\begin{center}\includegraphics[scale=.8]{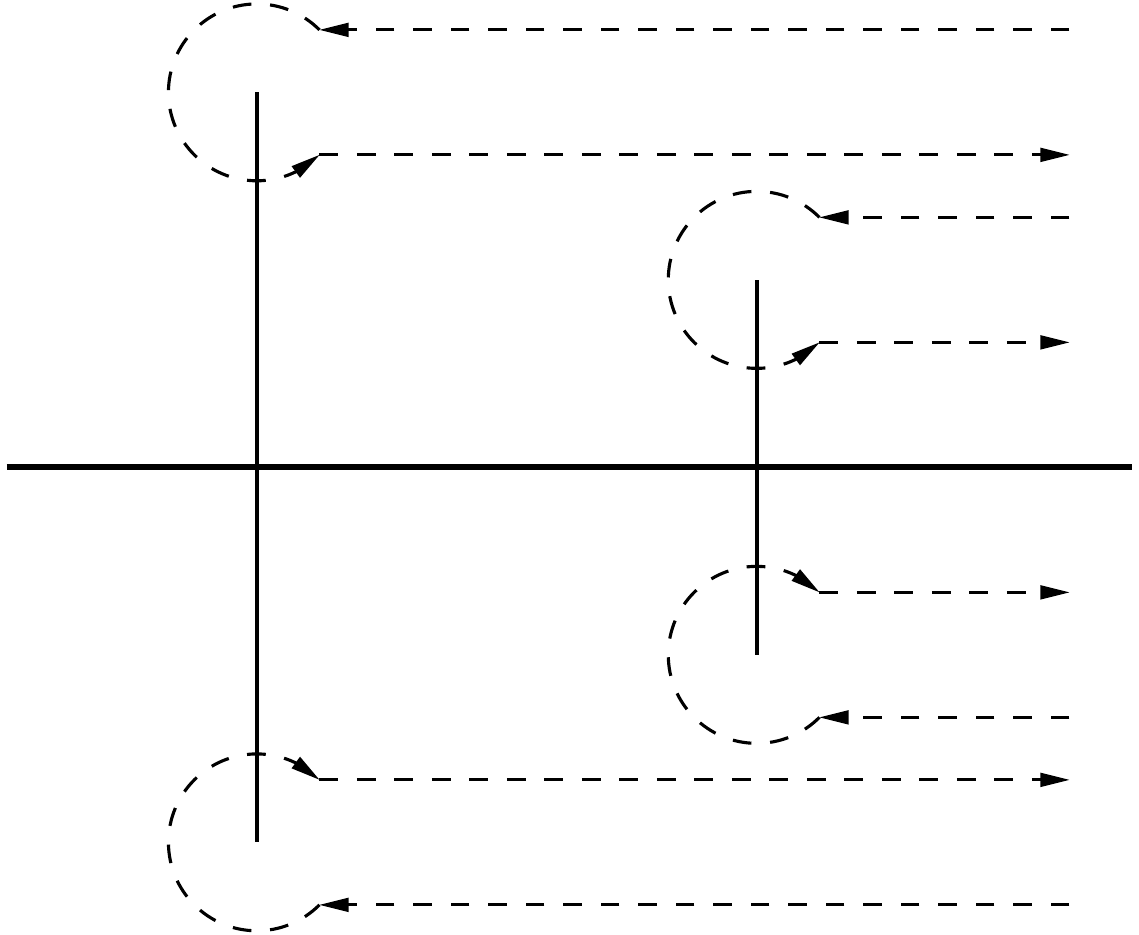}\end{center}
\caption{Branch cuts symmetric w.r.t. $\R$ (vertical lines) and corresponding loops (dashed).}\label{Fig:cut}
\end{figure}

Then $\pi_1(\hat \C\setminus\{\lambda_1,\dots,\lambda_{\bar 1}\})$ is the group with generators $\gamma_1,\dots,\gamma_{\bar 1}$ and relation
$$\gamma_1\dots\gamma_n\gamma_{\bar n}^{-1}\dots\gamma^{-1}_{\bar 1}=1$$
To a representation $\rho:\pi_1(\H\setminus\{z_1,\dots,z_n\})\rightarrow \SL_2(\R)$, we associate the representation $\tilde\rho:\pi_1(\hat\C\setminus\{\lambda_1,\dots,\lambda_{\bar 1}\})\rightarrow \SL_2(\R)$ specified by
$$\tilde\rho(\gamma_k)=\tilde\rho(\gamma_{\bar k})=\rho(\gamma_i)$$
for $k=1,\dots,n$.

Let us consider again the Schlesinger equation \eqref{eq:Schlesinger}. We denote by 
$$\Lambda=\{(\lambda_1,\dots,\lambda_{\bar 1})\in\H^n\times (-\H)^n: \lambda_i\neq\lambda_j{\rm\ for\ }i\neq j,{\rm }\lambda_{\bar k}=\overline{\lambda_k}{\rm\ for\ all\ }k\}$$
a $2n$-dimensional real analytic space.

\begin{Lem}
If for some $\lambda_0\in\Lambda$, %
$\overline{A_{\bar k}}(\lambda_0)=A_k(\lambda_0)$ for all $k$, then under \eqref{eq:Schlesinger} these conditions are still satisfied for $\lambda\in\Lambda$ near $\lambda_0$.
\end{Lem}
\begin{proof}
This follows from observing that by conjugating \eqref{eq:Schlesinger}, 
$$(\overline{A_{\bar 1}},\dots,\overline{A_{\bar n}},\overline{A_n},\dots,\overline{A_1})$$
gives another solution restricted to $\Lambda$ with the same initial condition and one concludes by connexity.
\end{proof}
This gives a reduction to $n$ $2\times 2$ matrices and $n$ - complex - variables (rather than $2n$ matrices and $2n$ variables).

Then we extend the parameter space $\Lambda$ to the following boundary component:
$$\Lambda_\R=\{(x_1,\dots,x_n,x_n\dots,x_1): x_1<\cdots<x_n\}\subset\R^{2n}$$

\begin{Lem}\label{Lem:fuchsbound}
If $N_1,\dots,N_n$ are real nilpotent matrices in $M_2(\R)$, there is a solution of \eqref{eq:Schlesinger} extending continuously to $\Lambda_\R$ s.t. $A_k=-\frac{1}{2i\pi}N_k=-A_{\bar k}$ on $\Lambda_\R$. The corresponding representation $\tilde\rho$ of $\pi_1(\hat\C\setminus\{\lambda_1,\dots,\bar\lambda_1\})$ is given by $\tilde\rho(\gamma_k)=\tilde\rho(\gamma_{\bar k})=\Id_2+N_k$ for $k=1,\dots,n$.
\end{Lem}
\begin{proof}
The only issue comes from terms of type $\partial_{\lambda_k}A_{\bar k}=[A_{\bar k},A_k]/(\lambda_{\bar k}-\lambda_k)$; the apparent singularity at $\lambda_k=\lambda_{\bar k}$ disappears since then $A_{\bar k}=-A_k$.

More precisely, we parameterize $\lambda_k=x_k+iy_k$, $\lambda_{\bar k}=x_k-iy_k$, $y_k\geq 0$, and $A_{\bar k}=\overline{A_k}$.
Then
\begin{align*}
\partial_{x_k}A_k&=(\partial_{\lambda_k}+\partial_{\lambda_{\bar k}})A_k=\partial_{\lambda_{\bar k}}A_k-\partial_{\lambda_k}A_{\bar k}-\sum_{j\neq k,\bar k}\partial_{\lambda_k}A_j\\
&=-\sum_{1\leq j\leq n,j\neq k}\left[\frac{A_j}{\lambda_j-\lambda_k}+\frac{\overline{A_j}}{\overline{\lambda_j}-\lambda_k},A_k\right]
\end{align*}
and 
\begin{align*}
\partial_{y_k}A_k&=-i(\partial_{\lambda_k}-\partial_{\lambda_{\bar k}})A_k=i\partial_{\lambda_{\bar k}}A_k+i\partial_{\lambda_k}A_{\bar k}+i\sum_{j\neq k,\bar k}\partial_{\lambda_k}A_j\\
&=\frac{[A_k,\overline A_k]}{y_k} +i\sum_{1\leq j\leq n,j\neq k}\left[\frac{A_j}{\lambda_j-\lambda_k}+\frac{\overline{A_j}}{\overline{\lambda_j}-\lambda_k},A_k\right]
\end{align*}
We want to show that this extends continuously to $y_k=0^+$ provided $[A_k,\overline{A_k}]$ is small enough for some small $y_k$. Up to relabelling we may assume $k=1$.
Let $K_0$ be a compact neighborhood of $\{(\lambda_j)_{j>1}\}$ (with $\lambda_{j'}\neq\lambda_j$ for $j\neq j'$ in $K$) and 
\begin{align*}
M(y)&=\sup_{[a+iy,b+iy]\times K_0}\{\sum_j\|A_j\|\}\\
m(y)&=\sup_{[a+iy,b+iy]\times K_0}\{\|A_1+\overline A_1\|\}
\end{align*}

We have
\begin{align*}
\partial_{y_1}A_1&=\frac{[A_1,A_1+\overline{A_1}]}{y_1}+(reg)
&\partial_{x_1}A_j&=(reg)\\
\partial_{y_1}(A_1+\overline{A_1})&=(reg)
&\partial_{x_1}(A_1+\overline{A_1})&=[(reg),A_1+\overline{A_1}]+y_1(reg)
\end{align*}
where the $(reg)$ have smooth coefficients up to $y_1=0$.

Then we estimate for $0<y<y'$
\begin{align*}
M(y')\leq&M(y)+k\int_{y}^{y'}(\frac{m(u)}u+M(u))M(u)du\\
m(y')\leq&m(y)+k\int_{y}^{y'}M(u)^2du
\end{align*}
with $k$ a constant depending on $K_0$. %
Then start from $y$ small and assume that $m(y)\leq cy$. Then for $C>0$ there is $c'>0$ s.t. $m(y')\leq c'y'$ as long as $M(y')\leq C$; from the first line this holds up to some $y'$ small but independent of $y$.

Consider a sequence of solutions $A_j^{(n)}$ s.t. $A_1^{(n)}(a+iy_n,\lambda_2^0,\dots)=(2i\pi)^{-1}M$, $M$ a fixed real matrix, and $A_j^{(n)}(a+iy_n,\lambda_2^0,\dots)$ are fixed matrices. Then the earlier estimates show that the sequence $A_j^{(n)}$ is equicontinuous in $([a,b]\times[0,y_0])\times K_0$ and consequently one can extract a subsequence converging to a solution of \eqref{eq:Schlesinger} which has a H\"older continuous extension up to $y_1=0^+$.

By isomonodromy, one can compute $\rho$ in the case $y_1=\cdots=y_n=\eps\ll 1$. In this case one can solve \eqref{eq:Fuchs} on a contour running along the real line except for semicircles around the $x_j$'s. If $\eps_0$ is small but fixed, $Y_0(x_j-\eps)=\Id_2+O(\eps)$ (along a path staying $\eps_0$-away from punctures); then one solves on the segment $[x_j-\eps_0,x_j+\eps_0]$, and then back to infinity staying away from the singularities. So the monodromy $\rho(\gamma_j)$ is within $O(\eps)$ of the monodromy when there is only one pair of punctures $x_j\pm i\eps$. In that case we solve explicitly
$$Y_0(z)=\left(\frac{z-\lambda_1}{z-\overline{\lambda_1}}\right)^{A_j}=
\exp\left(A_j\log\frac{z-\lambda_j}{z-\overline{\lambda_j}}\right)$$
so that $\rho(\gamma_j)=\exp(-{2i\pi}A_j)$. Remark that $\exp(N)=\Id_2+N$ for $N$ a $2\times 2$ nilpotent matrix.
\end{proof}
Throughout the solution of \ref{eq:Schlesinger} is fixed by this boundary condition on $\Lambda_\R$; in particular the $A_j$'s are always nilpotent. It follows that the fundamental solution (with $Y_0(\infty)=\Id_2$) of \eqref{eq:Fuchs} is $O(|\log(z-\lambda_i)|)$ near punctures. As discussed earlier, it is the unique ($\SL_2(\C)$-valued) such solution of the corresponding RH problem.

Then we can define consider the $\tau$-function defined up to multiplicative constant by \eqref{eq:dlogtau}; it is a function of the position of the punctures $\lambda_i$ and of the unipotent matrices $\rho(\gamma_i$ (and the homotopy classes of the $\gamma_i$'s).

\begin{Lem}\label{Lem:taubound}
The $\tau$-function extends to continuously to $\Lambda_\R$ and is constant there.
\end{Lem} 
\begin{proof}
Similarly to Lemma \ref{Lem:fuchsbound}, when solving \ref{eq:dlogtau} near $\Lambda_\R$, the only potential issue comes from the term
$$\Tr(A_iA_{\bar i})\frac{d(\lambda_i-\lambda_{\bar i})}{\lambda_i-\lambda_{\bar i}}$$
From Lemma \ref{Lem:fuchsbound} we know that $A_i$ has a H\"older continuous extension to the boundary and $A_{\bar i}=\overline{A_i}=-A_i$ on the boundary (viz. when $\lambda_i\in\R$). Since $A_i$ is nilpotent, $\Tr(A_iA_{\bar i})=0$ when $\lambda_i\in\R$ and
consequently
$$\frac{\Tr(A_iA_{\bar i})}{\Im(\lambda_i)}=O(\Im(\lambda_i)^{\alpha-1})$$
for some $\alpha>0$. It follows that $\tau$ itself has a finite limit as $\lambda_i\rightarrow\R$ which is locally constant in $\Re(\lambda_i)$.
\end{proof}
Consequently we can fix the multiplicative constant s.t. $\tau\equiv 1$ on $\Lambda_\R$.

Then we proceed to check that $\tau$ is essentially independent of choices, viz. the ordering of the punctures and the generators $\gamma_1,\dots,\gamma_n$ (the fundamental group of $\Lambda$ is a pure braid group; solving \eqref{eq:dlogtau} along non-contractible loops in the configuration space $\Lambda$ a priori generates a character of that group). See Figure \ref{Fig:braid}.
\begin{figure}
\begin{center}\includegraphics[scale=.8]{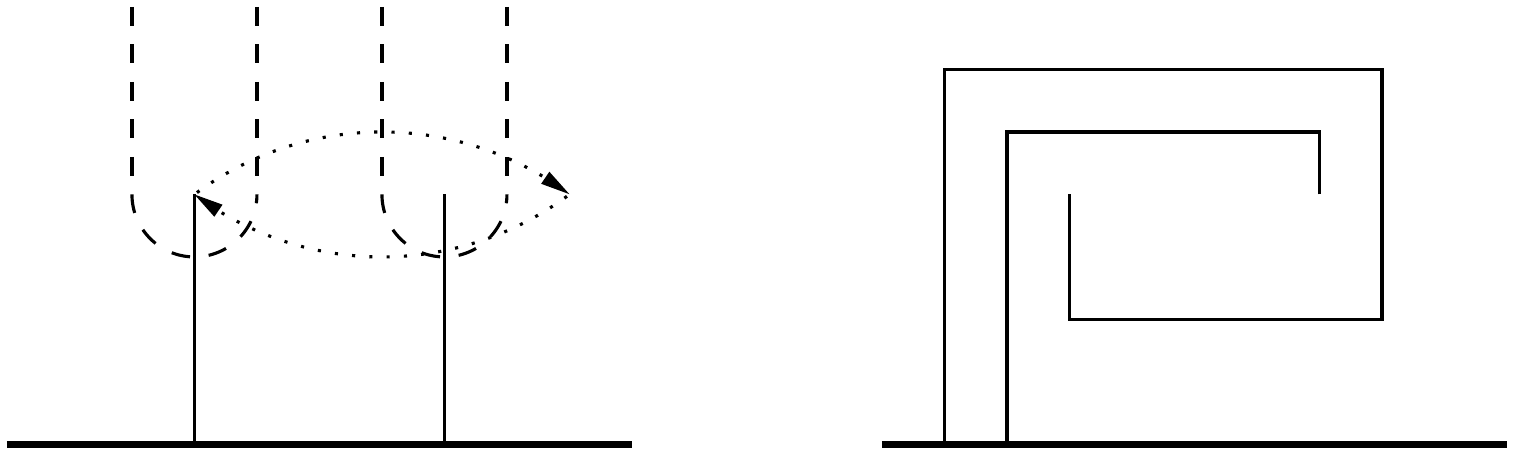}\end{center}
\caption{Effect of braiding on branch cuts. Left: branch cuts (solid) and loops (dashed) before braiding (the left puncture is moved along the dotted cycle). Right: deformed branch cuts after braiding.}\label{Fig:braid}
\end{figure}

\begin{Lem}\label{Lem:tauchoice}
The $\tau$-function depends only on the set $\{\lambda_1,\dots,\lambda_n\}$ and the representation $\rho:\pi(\H\setminus\{\lambda_1,\dots,\lambda_n\})\rightarrow\SL_2(\R)$, up to conjugacy.
\end{Lem}
\begin{proof}
First from \eqref{eq:dlogtau} we observe that the logarithmic variation of $\tau$ depends on the $A$'s, which can be recovered from the fundamental solution $Y_0$, which is itself uniquely characterized by $\rho$ and the growth condition. Conjugating $\rho$ amounts to conjugating all the $A_j$'s by the same constant matrix and this does not affect $\tau$. So the logarithmic variation of $\tau$ depends only on the set of punctures and $\rho$, and we need only check that the character of the pure braid group it generates is trivial. 

It is enough to consider the case where one puncture (say $\lambda_1$) circles around another one (say $\lambda_2$) without encircling any other puncture, as these moves generate the fundamental group of $\Lambda$. We may also assume that all the punctures are close to $\R$ (by choosing a ``root" on the configuration space $\Lambda$). 

When only one singularity ($\lambda_1$) is away from the boundary and all the other ones are on the boundary and fixed, one gets a solution of the (limiting) Schlesinger equations \eqref{eq:Schlesinger} given by $A_1,A_{\bar 1}$ constant, $A_j=-A_{\bar j}$ for all $j$ and 
$$A_j(\lambda)=\left(\frac{\lambda_j-\lambda_1}{\lambda_j-\overline\lambda_1}\right)^{-A_1}N_j\left(\frac{\lambda_j-\lambda_1}{\lambda_j-\overline\lambda_1}\right)^{A_1}$$
for $j=2,\dots,n$. Since $A_1$ commutes with $z^{A_1}$, $\Tr(A_1A_j)=\Tr(A_1N_j)=-\Tr(A_1N_{\bar j})$ is constant (and $\Tr(A_1A_{\bar 1})=\Tr(-A_1^2)=0$ by nilpotency) and it follows that $\tau$ stays constant (with all but one punctures fixed and on the boundary; see also Corollary \ref{Cor:tau1} below).

Starting from $\lambda_1<\lambda_2$, we may proceed in the following way: move $\lambda_1$ along a clockwise semicircle to $\lambda'_1>\lambda_2$ (half-twist); then move $\lambda_2$ to $\lambda'_2>\lambda_1$; and translate $\lambda_1',\lambda_2'$ back to $\lambda_1,\lambda_2$. By Lemma \ref{Lem:taubound} and the previous argument, $\tau$ is unchanged by this sequence of operations, which is what we needed to check.
\end{proof}

The upper half-plane admits conformal automorphisms (homographies) than operate diagonally on $\Lambda$. If $\phi$ is such an homography, there is a natural identification $\pi_1(\H\setminus\{\lambda_1,\dots,\lambda_n\})\simeq \pi_1(\H\setminus\{\phi(\lambda_1),\dots,\phi(\lambda_n)\})$.

\begin{Lem}\label{Lem:taumoeb}
The $\tau$-function is invariant under homographies.
\end{Lem}
\begin{proof}
The group of homographies $\H\rightarrow\H$ is a 3-parameter Lie group and it is enough to check the statement infinitesimally, viz. we need to check
$$\sum_j\lambda_j^{m+1}\partial_{\lambda_j}\log\tau=0$$
for $m=-1,0,-1$. For $m=-1$ (translation) it is immediate from \eqref{eq:dlogtau}. For $m=0$ (scaling) we obtain
$$\sum_j\lambda_j^{m+1}\partial_{\lambda_j}\log\tau=\frac 12\sum_{i\neq j}\Tr(A_iA_j)=-\frac 12\sum_i\Tr(A_i^2)=0$$
by nilpotency, taking into account $\sum_jA_j=0$. (More generally $\tau$ is Moebius invariant when seen as a form with weights given in terms of the local monodromy exponents). Finally for $m=1$ we get
$$\sum_j\lambda_j^{2}\partial_{\lambda_j}\log\tau
=\frac 12\sum_{i\neq j}\Tr(A_iA_j)(\lambda_i+\lambda_j)=\sum_i\Tr(A_i\sum_j\lambda_jA_j)=\Tr((\sum_iA_i)(\sum_j \lambda_jA_j))=0
$$
using again $\Tr(A_i^2)=0$ and $\sum_iA_i=0$.
\end{proof}
Consequently, if $D$ is any simply-connected domain with punctures $\lambda_1,\dots,\lambda_n$ and $\rho:\pi_1(D\setminus\{\lambda_1,\dots,\lambda_n\})\rightarrow\SL_2(\R)$ is a representation with unipotent local monodromies, one can define unambiguously
$$\tau=\tau(D\setminus\{\lambda_1,\dots,\lambda_n\},\rho)$$
in such a way that if $\phi:D\rightarrow D'$ is a conformal equivalence,
$$\tau(\phi(D\setminus\{\lambda_1,\dots,\lambda_n\}),\rho)=\tau(D\setminus\{\lambda_1,\dots,\lambda_n\},\rho)$$
with the natural identification of the fundamental groups.

\begin{Cor}\label{Cor:tau1}
If $n=1$, $\tau(D\setminus\{\lambda\},\rho)=1$.
\end{Cor}
\begin{proof}
We may conformally map $D\setminus\{\lambda\}$ to $\H\setminus\{\eps i\}$ for $\eps>0$ arbitrarily small and use Lemma \ref{Lem:taubound}.
\end{proof}

Our next task is to study the limiting behavior of $\tau$ under pinching of the domain. Recall that a sequence of simply connected domains $(D_m)_{m\geq 0}$, containing a fixed point $\lambda$, Carath\'eodory-converges (seen from $\lambda$) to a simply connected domain $D$ if the uniformizing maps $\phi_m:\U\rightarrow D_m$ s.t. $\phi_m(0)=\lambda$, $\phi_m'(0)>0$ converge uniformly on compact subsets of the unit disk $\U$ to the uniformizing map $\phi:\U\rightarrow D$. 

We make a somewhat ad hoc modification of this classical notion in the presence of multiple marked bulk points (punctures). If $\mu_1,\dots,\mu_k,\mu_{k+1},\mu_{k+\ell}$ are other marked points, we say that $(D_m,\mu_1,\dots,\mu_{k+\ell})$ C-converges (seen from $\lambda$) to $(D,\mu_1,\dots,\mu_k)$ if $D_m$ C-converges to $D$, $\phi\circ\phi_m^{-1}(\mu_j)$ converges to $\mu_j\in D$ for $j\leq k$ and $\phi_m^{-1}(\mu_j)$ eventually exits any compact subset of $\U$ for $j=k+1,\dots,k+\ell$.

\begin{Lem}\label{Lem:pinch}
Let $D_m$ be a decreasing subsequence of simply-connected domains containing the $n$ distinct punctures $\lambda_1,\dots,\lambda_n$ s.t. $\cap_m D_m=D^-\sqcup D^+$, $D^\pm$ simply-connected, $\lambda_1,\dots,\lambda_k\in D^-$, $\lambda_{k+1},\dots,\lambda_n\in D^+$. A representation $\rho:\pi_1(D\setminus\{\lambda_1,\dots,\lambda_n\})\rightarrow\SL_2(\R)$ is given; it restricts to representations $\rho_\pm$ of $\pi_1(D^\pm\setminus\{\lambda_1,\dots,\lambda_n\})$.
Then
$$\lim_{m\rightarrow\infty}\tau(D_m\setminus\{\lambda_1,\dots,\lambda_n\},\rho)=\tau(D^-\setminus\{\lambda_1,\dots,\lambda_k\},\rho_-)\tau(D^+\setminus\{\lambda_{k+1},\dots,\lambda_n\},\rho_+)$$
\end{Lem}
Remark that $\rho_{\pm}$ are a priori defined up to conjugation, but that does not affect the value of $\tau$.
\begin{proof}
By the Carath\'erodory kernel theorem, the assumptions imply that
\begin{itemize}
\item seen from $\lambda_1$, $(D_m,\lambda_1,\dots,\lambda_n)$ C-converges to $(D^-,\lambda_1,\dots,\lambda_k)$, and that
\item seen from $\lambda_n$, $(D_m,\lambda_1,\dots,\lambda_n)$ C-converges to $(D^+,\lambda_{k+1},\dots,\lambda_n)$.
\end{itemize}
Let $\psi_m:D_m\rightarrow \H$ be the conformal equivalence s.t. $\psi_m(\lambda_1)=i$, $\psi_m(\lambda_n)=i+M_m$, $M_m\in (0,\infty)$. Then $M_m\rightarrow\infty$. Set $\lambda_j^m=\psi_m(\lambda_j)$. Then 
\begin{itemize}
\item $\lambda_1^m,\dots,\lambda_k^m$ converge to $k$ distinct points $\lambda^\infty_1,\dots,\lambda^\infty_k$ in $\H$.
\item $\lambda_{k+1}^m-M_m,\dots,\lambda_n^m-M_m$ converge to $n-k$ distinct points $\lambda^\infty_{k+1},\dots,\lambda^\infty_n$ in $\H$.
\end{itemize}
Then we want to show that
$$\lim_{m\rightarrow\infty}\tau(\H\setminus\{\lambda_1^m,\dots,\lambda_n^m\},\rho)=\tau(\H\setminus\{\lambda^\infty_1,\dots,\lambda_k^\infty\},\rho_-)
\tau(\H\setminus\{\lambda^\infty_{k+1},\dots,\lambda_n^\infty\},\rho_+)$$
By direct examination of \eqref{eq:Schlesinger} we obtain
\begin{align*}
A_j(\lambda_1^m,\dots,\lambda_n^m,\rho)&=A_j(\lambda_1^m,\dots,\lambda_k^m,\rho_-)+O(M_m^{-1})&{\rm\ for\ }1\leq j\leq k\\
A_j(\lambda_1^m,\dots,\lambda_n^m,\rho)&=A_j(\lambda_{k+1}^m,\dots,\lambda_{n}^m,\rho_+)+O(M_m^{-1})&{\rm\ for\  }k<j\leq n
\end{align*}
and then from \eqref{eq:dlogtau}
$$\log\tau(\lambda_1^m,\dots,\lambda_n^m,\rho)=\log\tau(\lambda_1^m,\dots,\lambda_k^m,\rho_-)+\log\tau(\lambda_{k+1}^m,\dots,\lambda_{n}^m,\rho_+)+O(M_m^{-1})$$
which concludes.
\end{proof}

\section{Convergence}

In this section we study the small mesh limit of Kenyon's double dimer observables (for real, locally unipotent representations). For perspective and comparison, we start with a brief discussion of the electric correlators analyzed in \cite{Dub_tors}, corresponding to unitary line bundles (above the punctured domain). Then we start discussing rank 2 bundles. (Remark however that the asymptotic analysis could be carried out {\em mutatis mutandis} for a general rank $r\geq 2$; but the combinatorial interpretation motivating this analysis seems specific to $r=2$). We start with a local analysis near a puncture (modeled on the case of the punctured sphere $\hat\C\setminus\{0,\infty\}$). Patching the local construction with the global information derived from the corresponding Fuchsian system, we estimate the inverting kernel of the (bundle) Kasteleyn operator. This then allows to evaluate the variation of double dimer correlators under a macroscopic isomonodromic deformation (away from the boundary). We conclude with {\em a priori} estimates for punctures near the boundary.

\subsection{The rank 1 case}

It may be useful for comparison and intuition to review some of the results of \cite{Dub_tors}, expressed in the $\tau$-function framework. We consider a dimer cover of a planar Temperleyan isoradial graph under the usual infinite volume measure. Punctures $\lambda_1,\dots,\lambda_n$ are located in faces of the graph and $\Sigma=\hat\C\setminus\{\lambda_1,\dots,\lambda_n\}$ is the punctured Riemann sphere. A representation $\rho:\pi_1(\Sigma)\rightarrow U_1(\C)$ is parameterized by $\chi_1,\dots,\chi_n\in\U$ with $\prod_j\chi_j=1$, so that $\rho(\ell_i)=\chi_i$ if $\ell_i$ is a ccwise oriented loop rooted at infinity separating $\lambda_i$ from the other punctures. Set $\chi_j=e^{2i\pi s_j}$ and assume furthermore that the $s_j$'s are small enough and $\sum_j s_j=0$. Let $V$ be the unitary line bundle over $\Sigma$ with global section $Y_0(z)=\prod_j (z-\lambda_j)^{s_j}$; this corresponds to $A_j=(s_j)$ (here the Schlesinger equations are of course trivial). Then
$$d\log\tau=\frac 12\sum_{i\neq j}s_is_j\frac{d(\lambda_j-\lambda_i)}{\lambda_j-\lambda_i}$$
so that $\tau\propto\prod_{i< j}(\lambda_j-\lambda_i)^{s_is_j}$ (a multivalued function on the configuration space). Set $S_\rho(z,w)=\frac{Y_0(w)^{-1}Y_0(z)}{z-w}$ and 
$$r_j=r^\rho_j=\lim_{z,w\rightarrow\lambda_j}\left(S_\rho(z,w)-\frac{((z-\lambda_j)/(w-\lambda_j))^{s_j}}{z-w}\right)$$
so that
$$d\log\tau=\sum_j s_j r_jd\lambda_j$$
Let us also consider the antiholomorphic line bundle $V^*$ with global section $Y_0^*(z)=\prod_j \overline{(z-\lambda_j)}^{-s_j}$ , whose sections are $\rho$-multivalued antiholomorphic functions. Plainly, $V^*=\overline{V_{\bar\rho}}$, and one may associate to this data an antiholomorphic $\tau$ function $\overline{\tau_{\bar\rho}}$.

In a discrete setting, consider a Temperleyan isoradial graph $\Xi$ with small mesh and punctures $\lambda_i$'s inside faces. Let $K_\rho:\C^{\Xi_B}\rightarrow\C^{\Xi_W}$ be the Kasteleyn operator twisted by $\rho$, and $K_\rho^{-1}$ the inverting kernel vanishing at infinity. Now displace $\lambda_1$ to a neighboring face of $\Xi$ (across the edge $(wb)$); let $\lambda'_1$ be the new position of the puncture and $K'_\rho$ be the corresponding operator. Then $K'_\rho K_\rho^{-1}$ is a rank 1 perturbation of the identity and
$$\det (K'_\rho K_\rho^{-1})=1+(\chi_1^{\pm 1}-1)K(w,b)K_\rho^{-1}(b,w)$$
($\pm$ depends on whether $w$ is to the left or to the right of $b$). A local analysis shows that $K_\rho^{-1}(b,w)$ may be expressed in terms of local data (viz. independent from $\lambda_j,s_j$, $j\neq1$), 
$r^\rho_1$ and $\overline {r^{\bar\rho}_1}$:
$$K(w,b)K_\rho^{-1}(b,w)=(\chi^{\pm 1}-1)^{-1}\left((u-1)+us_1\left(r_1^\rho\cdot (\lambda'_1-\lambda_1)+\overline{r^{\bar\rho}_1}\cdot\overline{(\lambda'_1-\lambda_1)}\right)\right)+o(\lambda'_1-\lambda_1)$$
where $u=\left(\frac{\lambda'_1-w}{\lambda_1-w}\right)^{-s_1}$ is an explicit unit number. This leads to
$$\det(K'_\rho K_\rho^{-1})= u\left(1+{s_1}\left(r_1^\rho\cdot (\lambda'_1-\lambda_1)+\overline{r^{\bar\rho}_1}\cdot\overline{(\lambda'_1-\lambda_1)}\right)\right)+o(\lambda'_1-\lambda_1)$$ 
where $\lambda_2,\dots,\lambda_n$ are regarded as fixed. In other words, $|\det(K_\rho K^{-1})|\simeq \tau_\rho\overline{\tau_{\bar\rho}}$ up to multiplicative constant.\\

\subsection{Rank 2}

Here we consider $\Xi=\Xi_\delta=\frac\delta 2(\Z\times\N)$ the upper half-plane square lattice with mesh $\delta$ (or a microscopic translation thereof, for reflection arguments); $\Xi=\Xi_B\sqcup\Xi_W$ is bipartite and we have a Kasteleyn operator $\rK:\R^{\Xi_B}\rightarrow\R^{\Xi_W}$, which we can duplicate to obtain an operator $\rK\oplus\rK:(\R^2)^{\Xi_B}\rightarrow(\R^2)^{\Xi_W}$. Given punctures $\lambda_1,\dots,\lambda_n$ in $\H$ and a representation $\rho:\pi_1(\H\setminus\{\lambda_1,\dots,\lambda_n\})\rightarrow\SL_2(\R)$ (and, for definiteness, a choice of branch cuts), we obtain a twisted operator $(\rK\oplus\rK)_\rho:(\R^2)^{\Xi_B}\rightarrow(\R^2)^{\Xi_W}$. For brevity we will simply denote it by $\rK_\rho$; for $\rho=\Id$ the trivial representation, we get the untwisted operator $\rK\otimes\rK$.

The goal here is to establish the following convergence result.

\begin{Lem}\label{Lem:conv}
Fix punctures $\lambda_1,\dots,\lambda_n$ in the upper half-plane and disjoint branch cuts. For $\rho:\pi_1(\H\setminus\{\lambda_1,\dots,\lambda_n\})$ a representation with unipotent local monodromies close enough to the identity,
$$\lim_{\delta\searrow 0}\det(\rK_\rho\rK_{\Id}^{-1})=\tau(\lambda_1,\dots,\lambda_n;\rho)$$
\end{Lem}
\begin{proof}
When the punctures are on the boundary, both sides are equal to 1. Moving the punctures at a small distance of the boundary changes both sides by a small amount (from Lemmas \ref{Lem:taubound}, \ref{Lem:nearbounddiscr}), uniformly in $\delta$ for $\delta$ small enough. Once the punctures are at small but macroscopic distance of the boundary, displacing them changes both sides by the same amount up to $o(1)$ by Lemma \ref{Lem:discrvar}, which concludes.
\end{proof}

\subsubsection{Single puncture}

We shall need a priori estimates on basic discrete harmonic and holomorphic functions with unipotent monodromies. The problems and general line of reasoning are rather similar to the case of unitary monodromies studied in \cite{Dub_tors}; however we need to modify the constructions and some of the arguments.

On the square lattice $\Z^2$, we consider the (positive) discrete Laplacian $\Lap:\R^{\Z^2}\rightarrow\R^{\Z^2}$ given by
$$\Lap f(x)=\sum_{y\sim x}((f(x)-f(y))$$
where $\sim$ designates adjacency of vertices. There is a {\em potential kernel} (eg Theorem 4.4.4 in \cite{LawLim_RW}) s.t. $\Lap a=-4\delta_0$ and
\begin{equation}\label{eq:potasymp}
\begin{array}{rll}
a(x)&=\frac 2\pi\log|x|+c+O(|x|^{-2})\\
a(x')-a(x)&=\frac 2\pi(\log|x'|-\log|x|)+O(|x|^{-3})&{\rm\ if\ }x\sim x'
\end{array}
\end{equation}
where $c$ is a (known but unimportant) constant.

Let consider the harmonic conjugate $a^*$ of $a$ defined on the dual lattice $(\Z^2)^*\simeq (\frac 12,\frac 12)+\Z^2$. It is such that
$$a^*(y')-a^*(y)=a(x')-a(x)$$
if $(xx')$ is an edge of $\Z^2$ and $(yy')$ is its dual edge (oriented so that $((xx'),(yy'))$ is a direct frame). Technically, $da^*$ is well defined a discrete 1-form which is closed except at $0$. Since $\Lap a=-4\delta_0$, $a^*$ is harmonic on $(\Z^2)^*$ and additively multivalued (with puncture at $0$): it increases by $4$ per counterclockwise turn around $0$. From \ref{eq:potasymp} and the fact that
$$\arg(y')-\arg(y)=\log|x'|-\log|x|+O(|x|^{-3})$$
if $(yy')$ is dual to $(xx')$, one deduces easily
\begin{equation}\label{eq:argasymp}
a^*(x)=\frac 2\pi\arg(x)+O(|x|^{-2})
\end{equation}
Note that $a^*$ is defined up to an additive constant. 

Let us now consider discrete holomorphic functions on the bipartite graph $\Diamond$ whose black vertices $\Diamond_B$ correspond to vertices of $(\Z^2)$ or its dual $(\frac 12,\frac 12)+\Z^2$; and white vertices $\Diamond_W$ are midpoints of edges of $\Z^2$. The Kasteleyn operator $\rK:\R^{\Diamond_B}\rightarrow\R^{\Diamond_W}$ is defined by
\begin{align*}
(\rK f)(w)&=\frac 12\left(f(w+\frac 12)-f(w-\frac 12)+f(w+\frac i2)-f(w-\frac i2)\right)&{\rm if\ }\Im(w)\in\Z\\
(\rK f)(w)&=\frac 12\left(f(w+\frac i2)-f(w-\frac i2)+f(w-\frac 12)-f(w+\frac 12)\right)&{\rm if\ }\Re(w)\in\Z
\end{align*}
One checks \cite{Ken_domino_conformal} that $4\rK^t\rK$ restricts to the discrete Laplacian on $\Z^2$ and its dual. In particular, $\rK f=0$ in some region if $f_{|\Z^2}$ is harmonic there and $f_{|(\Z^2)^*}$ is its harmonic conjugate. It will be convenient to assign phases to vertices of $\Diamond$ in the following way
\begin{align*}
e^{i\nu(b)}&=1&{\rm if\ }b\in\Z^2\\
e^{i\nu(b)}&=i&{\rm if\ }b\in(\Z^2)^*\\
e^{i\nu(w)}&=1&{\rm if\ }w\in \Diamond_W, \Im(w)\in\Z\\
e^{i\nu(w)}&=i&{\rm if\ }w\in \Diamond_W, \Re(w)\in\Z
\end{align*}
We may write $(\rK f)(w)=\sum_{b\sim w}\rK(w,b)f(b)$, with the matrix element $\rK(w,b)$ given by
$$\rK(w,b)=(b-w)e^{-i\nu(b)-i\nu(w)}$$

As explained in \cite{Ken_domino_conformal}, one can construct an inverting kernel for $\rK$ from the potential kernel $a$. If $w$ corresponds to the edge $(xx')$ of $\Z^2$ (oriented eastward or northward, so that $x'-x=e^{i\nu(w)}$), one considers the function $\frac 12(a(.,x')-a(.,x))$ which is harmonic on $\Z^2$ except at $x',x$. Its harmonic conjugate (vanishing at infinity) is single-valued. Taken together they define a function $f$ on $\Diamond_B$ which is discrete holomorphic except at $w$; we denote it by $\uK^{-1}(.,w)$. Then
\begin{equation}\label{eq:Cauchyplane}
\begin{array}{cl}
\rK\uK^{-1}(.,w)&=\delta_w\\
\uK^{-1}(b,w)&=\Re\left(\frac{e^{i(\nu(b)+\nu(w))}}{\pi(b-w)}\right)+O(|b-w|^{-3})
\end{array}
\end{equation}

We want to construct a similar inverting kernel in the presence of a unipotent monodromy. More precisely, fix a $2\times 2$ unipotent matrix $P$ and a face $f$ of $\Diamond$. Let us consider the lift $\tilde\Diamond$ of the graph $\Diamond$ to the universal cover of $\C\setminus \{f\}$. One can describe $\tilde\Diamond$ in terms of decks indexed by $\Z=\pi_1(\C\setminus\{ f\})$ isomorphic to $\Diamond$ cut along a branch cut from $f$ to $\infty$; upon traversing a branch cut one moves up or down a deck. Remark that $\rK$ has a natural lift as an operator $\R^{\tilde\Diamond_B}\rightarrow\R^{\tilde\Diamond_W}$.

Let $\theta$ be the deck transformation of $\tilde\Diamond$ associated to a counterclockwise loop around $0$. We consider the space
$$(\R^2)^{\Diamond_\cdot}_P=\{f\in(\R^2)^{\Diamond_.}: f\circ\theta=fP\}$$
with $\cdot\in\{W,B\}$, which is identified with $P$-multivalued functions on $\Diamond$ punctured at $f$. (Here we consider $f$ taking values in $1\times 2$ row vectors). We can then define $\rK:(\R^2)^{\Diamond_B}_P\rightarrow(\R^2)^{\Diamond_W}_P$. 

Up to conjugation, we may choose $P=\left(\begin{array}{cc}1&1\\0&1\end{array}\right)$.  For definiteness let us a fix a branch cut (a simple path from $\infty$ to $f$ on $\Diamond^*$), which we can take to be a straight half-line ending at the puncture. For each $b\in\Diamond_B$ at the end of an edge crossing (from left to right) the branch cut, we separate $b$ into two vertices $b^{\pm}$ - the endpoint of the crossing edge - and $b^{\mp}$ - the endpoint of other edges initially adjacent to $b$, in such a way that $b^+$ (resp. $b^-$) is connected to the right (resp. left) hand of the cut, oriented toward the puncture. The set of such vertices $b^\pm$ is denoted by $\partial$ and the planar graph obtained in this fashion is $\bar\Diamond$. Then we have the identification 
$$(\R^2)^{\Diamond_B}_P\simeq \{f\in(\R^2)^{\bar\Diamond_B}: f(b^+)=f(b^-)P{\rm\ \ }\forall b^\pm\in\partial\}$$
(by distinguishing a zeroth deck). This may be seen as natural discretization of a Riemann-Hilbert problem.

In the standard case (ie in the absence of monodromies), discrete holomorphic functions vanishing at infinity vanish identically. More precisely, we have a maximum principle: the uniform norm in a domain of a discrete holomorphic function is less than its norm on the boundary of that domain. The following lemma is a substitute for that statement in the presence of unipotent monodromies.

\begin{Lem}\label{Lem:Pmax}
Let $f\in(\R^2)^{\Diamond_B}_P$ be s.t. $\rK f=0$ in $B(0,R)$. Then for $b\in\overline\Diamond_B\cap B(0,R)$,
$$f(b)=O((1+\log|R/b|)\|f_{|\gamma}\|_\infty)$$
where $\gamma$ consists of two simple paths on $\Gamma,\Gamma^*$ respectively within $O(1)$ of the circle $C(0,R)$.
\end{Lem}
\begin{proof}
The first coordinate $f_1$ of $f$ is single-valued and restricts to discrete harmonic functions on $\Gamma,\Gamma^*$ and the result follows immediately from the maximum principle.

The second coordinate $f_2$ restricts to multivalued discrete harmonic functions on $\Gamma,\Gamma^*$, which increase by $f_1$ for each (counterclockwise) turn around $0$. So let us start a random walk $X$ on $\Gamma$ from $b$ on the zeroth-deck. We have
$$f_2(b)=\E(f_2(X_n)+N_nf_1(X_n))$$
where $N_n$ is the deck index at time $n$. We need to estimate moments of $N_n$ up to the time $\tau$ of first exit of $B(0,R)$. For this we use the discrete argument \eqref{eq:argasymp}. More precisely, $n\mapsto a^*(X_n)$ is a martingale which is within $O(1)$ of $4N$; its increments are $O(1/X_n)$. Consequently we have the $L^2$ estimate
$$\E((a^*(X_{n\wedge\tau})-a^*(X_0))^2)=\sum_{k=0}^{n-1}\E((a^*(X_{(k+1)\wedge\tau})-a^*(X_{k\wedge\tau}))^2)\leq c\E(\sum_{k=0}^{\tau-1}\frac{1}{1+|X_k|^2})\leq c\E\left(\sum_{x\in\Gamma} G(b,x)\frac{1}{1+|x|^2}\right)$$
where $G$ is the Green kernel for the random walk killed on the boundary $\gamma$. Basic random walk estimates give
$$G(b,x)=O\left(1+\log\frac {R-|x|}{|b-x|}\right)$$
if $|b-x|\leq R-|x|$ and $G(b,x)=O(\frac{R-|x|}{b-x})$ otherwise, which leads (together with a $L^2$ maximal inequality for martingales) to 
$$\E(\sup_{n\leq\tau}a^*(X_n)^2)=O((\log(R/|b|)^2)$$
One concludes by dominated convergence and Cauchy-Schwarz.
\end{proof}

We are interested in the behavior of $P$-multivalued discrete holomorphic functions near the singularity and for that purpose we shall need discrete logarithms. Let us consider the discrete holomorphic function $\Log:{\overline\Diamond}_B\rightarrow\R$ given by
$$\Log=\frac\pi 2\left(\ind_{\Z^2}(a-c)+\ind_{(\Z^2)^*}a^*\right)$$
so that 
\begin{align*}
\Log(b)&=\Re(e^{i\nu(b)}\log(b))+O(|b|^{-2})\\
\Log(b^+)-\Log(b^-)&=2\pi\ind_{\partial\cap(\Z^2)^*}&\forall b\in\partial
\end{align*}
By exchanging the roles of the two dual square lattices, one defines similarly a discrete holomorphic function $\Log^*{\overline\Diamond}_B\rightarrow\R$ s.t.
\begin{align*}
\Log^*(b)&=\Re(e^{i\nu(b)}i\log(b))+O(|b|^{-2})\\
\Log^*(b^+)-\Log^*(b^-)&=-2\pi\ind_{\partial\cap(\Z^2)}&\forall b\in\partial
\end{align*}

\begin{Cor}
If $f$ is discrete holomorphic and $P$-multivalued in $B(0,R)$ and if 
$$\left|f(b)-\Re\left(e^{i\nu(b)}%
\left(\alpha,\beta+\gamma\log(b)\right)
\right)\right|\leq\eps$$
on $\gamma$, then
$$\left|f(b)-\Re\left(e^{i\nu(b)}%
(\alpha,\beta)
\right)-\left(%
0,\Re(\gamma)\Log(b)+\Im(\gamma)\Log^*(b)
\right)\right|=O(\eps(1+\log(R/|b|)))$$
in $B(0,R)$.
\end{Cor}

We now want to construct and estimate $\rK_P^{-1}$, an inverting kernel for $\rK$ operating on $(\R^2)^{\Diamond_B}_P$. We observe that
$$f=%
(0,\uK^{-1}(.,w))
$$
is $P$-multivalued and satisfies $\rK f=%
(0,\delta_w)$. 
Let us fix a white vertex $w\in\Diamond_W$ and consider the problem: find $g\in(\R^2)^{\bar\Diamond_B}$ discrete holomorphic and with jump condition
\begin{align}\label{eq:jump}
g(b^+)-g(b^-)=\uK^{-1}(b,w)&&\forall b\in\partial
\end{align}
Given such a $g$, we may consider
$$f=%
\left(\uK^{-1}(.,w), g+n\uK^{-1}(.,w)\right)$$
on the $n$-th deck; this is $P$-multivalued and satisfies $\rK f=\delta_w$ on the zeroth deck. So we are left with solving \eqref{eq:jump}.

Up to translation and rotation we may assume $f=\frac{-1+i}2$, $\Re(w)\geq 0$ and $\partial=(-\N)\cup(-\frac{1+i}2-\N)$ (corresponding to black vertices on $\Z^2$ and its dual).
From there we obtain the explicit representation
\begin{align*}
2\pi g(b)&=-\sum_{n=0}^\infty\uK^{-1}(-n,w)(\Log(b+n)-\Log(b+n+1))+\sum_{n=0}^\infty\uK^{-1}(-n-\frac{1+i}2,w)(\Log^*(b+n)-\Log^*(b+n+1))
\end{align*}
Then one obtains the asymptotics
$$g(b)=\Re\left(e^{i(\nu(b)+\nu(w))}\frac{\log(b/w)}{\pi(b-w)}\right)+O(|w|^{-2})$$
when $|w|,|b|$ and $|b-w|$ are comparable (and the branch of $\log$ is chosen w.r.t. the branch cut); and the estimate
\begin{align*}
g(b)=O\left(\sum_{k\geq 0}\frac{1}{(k+|w|)(k+|b|)}\right)&=O(|w|^{-1}\log|w/b|)&{\rm if\ }|b|<|w|/2\\
&=O(|b|^{-1}\log|b/w|)&{\rm if\ }|b|>2|w|
\end{align*}
Let us summarize the previous discussion.

\begin{Lem}\label{Lem:1punct}
Let $P=\left(\begin{array}{cc}1&1\\0&1\end{array}\right)$ and $w\in\Diamond_W$. There is a unique matrix-valued function $h(w)\in (M_2(\R))^{\overline\Diamond_B}$ s.t.:
\begin{enumerate}
\item $\rK h=\delta_w\Id_2$.
\item $h(b^+)=h(b^-)P$ on $\partial$.
\item For $b,w,b-w$ comparable,
$$h(b)=\Re\left(\frac{e^{i(\nu(b)+\nu(w))}}{\pi(b-w)}
\left(\begin{array}{cc}
1&\frac{\log(b/w)}{2i\pi}\\
0&1\end{array}\right)\right)+O(1/|w|^2)$$
\item $h(b)=O(|w|^{-1}\log|w/b|)$ for $|w|>2|b|$ and $h(b)=O(|b|^{-1}\log|b/w|)$ for $|b|>2|w|$.
\end{enumerate}
This function is denoted by $\uK_{P}^{-1}(.,w)$.
\end{Lem}
For a general unipotent matrix $Q=GPG^{-1}$ (distinct from $\Id_2$), we set
$$\uK^{-1}_Q=\uK^{-1}_PG$$

\subsubsection{Parametrix and inverting kernel}

Our goal is now to estimate the inverting kernels for discrete rank 2 bundles with unipotent local monodromies, in particular near the singularities. The basic building blocks will be the fundamental solution of the corresponding continuous problem \eqref{eq:Fuchs} away from the singularities, and the discrete kernels constructed above near the singularities (Lemma \ref{Lem:1punct}).

Let $\Gamma=\Gamma_\delta=\delta\left(\Z\times(\frac 12+\Z)\right)$; $\delta>0$ is a small scaling parameter (mesh size); let $\Gamma^*=\delta((\frac 12+\Z)\times\Z)$ and $\Diamond$ the bipartite graph obtained by superimposing $\Gamma$ and $\Gamma^*$: black vertices of $\Gamma$ correspond to vertices of $\Gamma$ or $\Gamma^*$ and white vertices of $\Xi$ correspond to (midpoints of) edges of $\Gamma$ or $\Gamma^*$. We define a Kasteleyn operator as before by
\begin{align*}
(\rK f)(w)&=\frac \delta2\left(f(w+\frac 12)-f(w-\frac 12)+f(w+\frac i2)-f(w-\frac i2)\right)&{\rm if\ }\Im(w)\in \delta(\frac 12+\N)\\
(\rK f)(w)&=\frac \delta 2\left(f(w+\frac i2)-f(w-\frac i2)+f(w-\frac 12)-f(w+\frac 12)\right)&{\rm if\ }\Re(w)\in\delta \Z
\end{align*}
Up to scale and shift, this is the set-up considered in the previous subsection. In particular we have an inverting kernel $\uK^{-1}$ with $\rK\uK^{-1}(.,w)=\delta_w$ and
$$\uK^{-1}(b,w)=\Re\left(\frac{e^{i(\nu(b)+\nu(w))}}{\pi(b-w)}\right)+O(\frac{\delta}{(b-w)^2})$$

Let $\Xi$ be the restriction of $\Diamond$ to vertices with (strictly) positive imaginary part. If $f:\Xi_B\rightarrow\R$ is discrete holomorphic, it can be extended to a discrete holomorphic function on $\Diamond_B$ by setting $f(\overline b)=f(b)$ for $b\in\Xi\cap\Gamma$, $f(\overline b)=-f(b)$ for $b\in\Xi\cap\Gamma^*$ and $f(b)=0$ on $\Diamond_B\cap\R$. This is a simple discrete version of classical Schwarz reflection arguments. This also shows that bounded discrete holomorphic functions on $\Xi$ are constant on $\Gamma$ and vanish on $\Gamma^*$. Note also that a discrete holomorphic function on $\Xi$ restricts to a discrete harmonic function with Neumann conditions on $\Gamma$, and to a discrete harmonic function with Dirichlet conditions on $\Gamma^*$. 

Similarly, one can define an inverting kernel for $\rK:\Xi_B\rightarrow\Xi_W$ by
$$\uK^{-1}_\Xi(b,w)=\uK^{-1}(b,w)+e^{-2i\nu(w)}\uK^{-1}(b,\bar w)=\Re\left(\frac{e^{i\nu(b)}}\pi\left(\frac{e^{i\nu(w)}}{b-w}+\frac{e^{-i\nu(w)}}{b-\bar w}\right)\right)+O(\frac{\delta}{(b-w)^2})$$

Now we introduce punctures $\lambda_1,\dots,\lambda_n$ (located at the center of faces of $\Xi$) and a representation $\rho:\pi_1(\C\setminus\{\lambda_1,\dots,\lambda_n\})\rightarrow\SL_2(\R)$ with unipotent local monodromies and close enough to the trivial representation. As in the single puncture case, there is a natural notion of $\rho$-multivalued (vector- or matrix-) valued functions on $\Xi_B$ or $\Xi_W$, which we denote by $(\R^2)^{\Xi_.}_\rho$, with $.\in\{B,W\}$. More precisely, one can lift $\Xi$ to the universal cover of $\H\setminus\{\lambda_1,\dots,\lambda_n\}$ to obtain a bipartite graph $\tilde\Xi$. Then we consider 
$$(\R^2)^{\Xi_.}_\rho=\{f\in(\R^2)^{\tilde\Xi_.}: f\circ\theta(\gamma)=f\rho(\gamma)\}$$
where $\theta(\gamma)$ is the deck transformation corresponding to the (rooted) loop $\gamma$. We have a natural Kasteleyn operator
$$\rK:(\R^2)^{\Xi_B}_\rho\rightarrow(\R^2)^{\Xi_W}_\rho$$
Alternatively, one can choose cuts $\delta_1,\dots,\delta_n$ from $\lambda_1,\dots,\lambda_n$ to the boundary and consider a cut graph $\bar\Xi$ (duplicating black vertices along the cut). Then $(\R^2)^{\Xi}_\rho$ is identified with functions in $(\R^2)^{\bar\Xi}_\rho$ with prescribed Riemann-Hilbert conditions along the cuts.

We want to find and estimate a kernel $\rK_\rho^{-1}$ s.t. $\rK^{-1}_\rho(.,w)$ is $\rho$-multivalued and $\rK\rK^{-1}_\rho(.,w)=\delta_w\Id_2$ (on a given deck/cut domain). In the style of \cite{Dub_tors}, we construct a parametrix $\rS_\rho$ by patching the limiting continuous kernels away from the singularity with the basic kernels (with 0 or one puncture) described above.

As before, we duplicate punctures by $\lambda_{\bar j}=\overline{\lambda_j}$ and extend $\rho$ by $\rho(\bar\gamma)=\rho(\gamma)$. Consider $Y_0$ the fundamental solution (normalized by $Y_0(\infty)=\Id_2$ of \eqref{eq:Fuchs} with punctures $\lambda_1,\dots,\lambda_{\bar 1}$ and monodromy representation given by $\rho$. (Again for definiteness, we can extend the cuts $\delta_1,\dots,\delta_n$ symmetrically across $\R$ and consider the corresponding Riemann-Hilbert  problem). In the continuum, we consider
$$S^\nu_\rho(z,w)=e^{i\nu}\frac{Y_0(w)^{-1}Y_0(z)}{\pi(z-w)}+e^{-i\nu}\frac{Y_0(\bar w)^{-1}Y_0(z)}{\pi(z-\bar w)}=e^{i\nu}\frac{\Id_2}{\pi(z-w)}+R^\nu(w)+O(z-w)$$
so that $S_\rho(.,w)$ has the correct monodromy, vanishes at $\infty$ and is real along $\R$; and 
$$R^\nu(w)=\pi e^{i\nu}A(w)+e^{-i\nu}\frac{Y_0(\bar w)^{-1}Y_0(w)}{w-\bar w}.$$
Notice that $S^{\pi/2}_\rho\neq iS^0_\rho$, due to the $\R$-linear boundary conditions.

{\bf Case: $w$ away from punctures and $\R$}
Assume that $w$ is at distance at least $\eta=\eta(\delta)$ of punctures ($\eta$ a mesoscopic scale to be specified) and $\R$. Set
\begin{align*}
\rS_\rho(b,w)&=\Id_2\uK^{-1}(b,w)+\Re(e^{i\nu(b)}R^{\nu(w)}(w))&{\rm if\ }|b-w|\leq\eta/3\\
&=\Re(e^{i\nu(b)}S^{\nu(w)}_\rho(z,w))&{\rm if\ }|b-w|\geq\eta/3,|b-\lambda_i|\geq\eta^3\\
&=0&{\rm\ otherwise}
\end{align*}
We estimate
$$\|\rK \rS_\rho(.,w)-\Id_2\delta_w\|_1\leq C(\eta(\eta+\delta\eta^{-2})+\delta+\eta^2\log(\eta))$$
where $\|.\|_1$ is the $L_1$ norm w.r.t. counting measure. The errors correspond respectively to the patching near the singularity $w$, the continuous approximation at macroscopic scale, and the patching near the punctures.

{\bf Case: $w$ near $\R$.}
If $\Im(w)=r\leq\eta$, we expand
$$S_\rho^\nu(z,w)=\Id_2\left(\frac{e^{i\nu}}{\pi(z-w)}+\frac {e^{-i\nu}}{\pi(z-\bar w)}\right)+R^\nu_1(w)+O(z-w)$$
and set
\begin{align*}
\rS_\rho(b,w)&=\Id_2(\uK^{-1}(b,w)+e^{-2i\nu(w)}\uK^{-1}(b,\bar w))+\Re(e^{i\nu(b)}R^{\nu(w)}_1(w))&{\rm if\ }|b-w|\leq 2\eta\\
&=\Re(e^{i\nu(b)}S_\rho^{\nu(w)}(z,w))&{\rm if\ }|b-w|\geq\eta/3,|b-\lambda_i|\geq\eta^2\\
&=0&{\rm\ otherwise}
\end{align*}
then estimate again
$$\|\rK \rS_\rho(.,w)-\Id_2\delta_w\|_1\leq C(\eta(\eta+\delta\eta^{-2})+\delta+\eta^2\log(\eta))$$
with errors coming from the patching near $w$, the continuous approximation in the bulk, and the patching near the punctures.

{\bf Case: $w$ near a puncture.}
If $w$ is at distance $r\leq\eta$ of a puncture $\lambda\in\{\lambda_1,\dots,\lambda_n\}$, we let %
$P$ be the jump matrix across the corresponding cut $\delta$. Up to gauge change (change of basis on the bundle) we may assume that $P=\left(\begin{array}{cc}1&1\\0&1\end{array}\right)$. Then
$$S_\rho^\nu(z,w)\left(\begin{array}{cc}1&-\frac{1}{2i\pi}\log((z-\lambda)/(w-\lambda))\\0&1\end{array}\right)$$
has a removable singularity at $\lambda$ and consequently one may expand
$$S_\rho^\nu(z,w)=\frac{e^{i\nu}}{\pi(z-w)}\left(\begin{array}{cc}1&\frac{1}{2i\pi}\log((z-\lambda)/(w-\lambda))\\0&1\end{array}\right)+R_1^\nu(w)+R_2^\nu(w)\log(z-\lambda)+O((z-\lambda)\log(z-\lambda))$$
Then we set
\begin{align*}
\rS_\rho(b,w)&=\rK^{-1}_P(b,w)+\Re(e^{i\nu(b)}R_1^{\nu(w)}(w))\\
&+\Re(R_2^{\nu(w)}(w))\Log(b-\lambda)+\Im(R_2^{\nu(w)}(w))\Log^*(b-w)&{\rm if\ }|b-\lambda|\leq 2\eta\\
&=\Re(e^{i\nu(b)}S_\rho^{\nu(w)}(z,w))&{\rm if\ }\min_i|b-\lambda_i|\geq\eta^2,|b-\lambda|>2\eta\\
&=0&{\rm\ otherwise}
\end{align*}
We estimate
$$\|\rK \rS_\rho(.,w)-\Id_2\delta_w\|_1\leq C(\eta(\eta+\delta\eta^{-2}+\eta)+\delta+\eta^2\log(\eta))$$
where the errors come from patching near $\lambda_1$, the continuous approximation error away from the singularities, and the patching error near $\lambda_i\neq\lambda$.

Consequently, if we set $\eta=\delta^{1/3}$, we have
$$\sup_w\|\rK \rS_\rho(.,w)-\Id_2\delta_w\|_1=O(\delta^{2/3}\log(\delta))$$

Classically, if $T$ is a kernel operator, say on a countable set $E$ w.r.t. counting measure, we have
$$\|Tf\|_1=\sum_{x\in E}|\sum_{y\in E}T(x,y)f(y)|\leq (\sup_y\|T(.,y)\|_1)\|f\|_1$$
Here, if we consider the operator $\rK\rS_\rho-\Id:(\R^2)^{\Xi_W}_\rho\rightarrow(\R^2)^{\Xi_W}_\rho$ (with a reference $L^1$ norm obtained from a choice of cuts), we obtain
$$\|T\|_{L^1\rightarrow L^1}=O(\delta^{2/3}\log(\delta))$$
where $\rK\rS_\rho=\Id+T$, and we can define an (exact) right inverse $\rK^{-1}_\rho$ of $\rK:(\R^2)^{\Xi_B}_\rho\rightarrow(\R^2)^{\Xi_W}_\rho$ by:
$$\rK^{-1}_\rho=\rS_\rho(\Id+T)^{-1}$$
Then
$$\rK^{-1}_\rho-\rS_\rho=-\rS_\rho T(\Id+T)^{-1}=-\rS_\rho T+\rS_\rho T^2(\Id+T)^{-1}$$
We are interested in particular in estimating $\rK^{-1}_\rho(b,w)$ for $b,w$ within $O(\delta)$ of a puncture $\lambda=\lambda_i$. By construction, if $b,w$ within $\eta^2$ of a puncture $\lambda$, $(\rS_\rho T)(b,w)=0$; indeed, $T(w',w)=0$ if $|w'-\lambda|\leq 2\eta$ and $\rS_\rho(b,w')=0$ if $|w'-\lambda|\geq\eta$. We have
$$\|T^2(\Id+T)^{-1}\delta_w\|_1=O(\delta^{4/3}(\log\delta)^2)$$
and $\sup_{w'}|\rS_\rho(b,w')|=O(\delta^{-1})$. We conclude:
$$|\rK^{-1}_\rho(b,w)-\rS_\rho(b,w)|=O(\delta^{1/3}(\log\delta)^2)$$
for $w,b$ within $O(\delta^{2/3})$ of a puncture.

\subsection{Variation}

We are now interested in the logarithmic variation of the determinant under displacement of a puncture. Let us single out a puncture, say $\lambda_1$, located in the center of a face of $\Xi$; $\lambda'_1$ is the center of an adjacent face. The other punctures and the monodromy representation $\rho$ are fixed. For definiteness we may consider a cut $\delta'_1$ from the boundary with last two vertices $\lambda_1,\lambda'_1$; $\delta_1$ denotes the sub-cut stopped at $\lambda_1$; cuts to other punctures are fixed. We let $\rK_\rho$ (resp. $\rK'_\rho$) be the Kasteleyn operator with jump at $\delta_1$ (resp. $\delta'_1$), so that $\rK'_\rho$ and $\rK_\rho$ differ by a rank 2 operator, and $\rK'_\rho\rK^{-1}_\rho$ differs from the identity by a rank $2$ operator. 

Let $(bw)$ be the edge of $\Xi$ crossed by $(\lambda_1\lambda'_1)$. Let $P$ be the unipotent matrix corresponding to the cut $\delta_1$. The $2\times 2$ block of $\rK'_\rho\rK_\rho^{-1}$ corresponding to $w$ is: 
$$\Id_2+\rK(w,b)\left(P^{\pm 1}-\Id_2\right)\rK_\rho^{-1}(b,w)$$
with $\pm 1=1$ (resp. $-1$) if $b$ to the right (resp. left) of $\delta_1'$, oriented from boundary to puncture.

We expressed $\rK^{-1}_\rho$ near a puncture $\lambda=\lambda_i$ in terms of the expansion of its continuous counterpart $S_\rho$ near $\lambda$. In order to establish the connection with $\tau$-functions, we want to expand $S_\rho$ near $\lambda$ in terms of the $A_j$'s of \eqref{eq:Fuchs}. Near $\lambda$ we expand
$$A(z)=\frac{A_i}{z-\lambda}+B+O(z-\lambda)$$
with $B=\sum_{j\neq i}\frac{A_j}{\lambda_j-\lambda_i}$. The leading behavior of the solution $Y$ of \eqref{eq:Fuchs} near $\lambda$ (normalized by $Y(w)=\Id_2$) is given by the solution of 
$$\left\{\begin{array}{rl}
Y_\lambda(w)&=\Id_2\\
\frac{d}{dz}Y_\lambda(z)&=\frac{A_i}{z-\lambda}Y_\lambda(z)\end{array}\right.$$
ie
$$Y_\lambda(z)=\exp\left({A_i}\log\frac{z-\lambda}{w-\lambda}\right)=\Id_2+{A_i}\log\frac{z-\lambda}{w-\lambda}$$
by nilpotency. The correction is given by Duhamel's formula:
\begin{align*}
Y(z)&=Y_\lambda(z)+\int_w^zY_\lambda(z)Y_\lambda(u)^{-1}BY_\lambda(u)du+O\left((z-w)^2\log^2\frac{z-\lambda}{w-\lambda}\right)\\
&=Y_\lambda(z)+\int_w^z\left(\Id_2+{A_i}\log\frac{z-\lambda}{u-\lambda}\right)B\left(\Id_2+{A_i}\log\frac{u-\lambda}{w-\lambda}\right)du+O\left((z-w)^2\log^2\frac{z-\lambda}{w-\lambda}\right)\\
\end{align*}
Taking into account
\begin{align*}
\int_1^x\log(t)dt&=x\log(x)-x+1\\
\int_1^x\log(x/t)dt&=-\log(x)+x-1\\
\int_1^x\log(t)\log(x/t)dt&=x\log(x)+\log(x)-2(x-1)
\end{align*}
we obtain
\begin{align*}
Y(z)&=Y_\lambda(z)+(z-w)\left([A_i,B]+B-2A_iBA_i\right)+((BA_i+A_iBA_i)(z-\lambda)\\
&+(-A_iB+A_iBA_i)(w-\lambda))\log\frac{z-\lambda}{w-\lambda}+O((z-w)^2\log^2(z-w))
\end{align*}
Since $YY_0(w)=Y_0$, we get for $|w-\lambda|\ll|z-\lambda|\ll 1$
\begin{align*}
S_\rho^\nu(z)&=e^{i\nu}\frac{Y_0(w)^{-1}Y_0(z)}{\pi(z-w)}+e^{-i\nu}\frac{Y_0(\bar w)^{-1}Y_0(z)}{\pi(z-\bar w)}=\left(e^{i\nu}\frac{Y_0(w)^{-1}}{\pi(z-w)}+e^{-i\nu}\frac{Y_0(\bar w)^{-1}}{\pi(z-\bar w)}\right)Y(z)Y_0(w)\\
&=\frac{e^{i\nu}}{\pi(z-w)}\left(\Id_2+Y_0(w)^{-1}A_iY_0(w)\log\frac{z-\lambda}{w-\lambda}\right)+\frac{e^{i\nu}}\pi Y_0(w)^{-1}([A_i,B]+B-2A_iBA_i)Y_0(w)\\
&+e^{-i\nu}\frac{Y_0(\bar w)^{-1}Y_0(w)}{\pi(w-\bar w)}+O((z-\lambda)\log(w-\lambda))
\end{align*}
which identifies
$$R^\nu_1(w)=\frac{e^{i\nu}}\pi Y_0(w)^{-1}([A_i,B]+B-2A_iBA_i)Y_0(w)+e^{-i\nu}\frac{Y_0(\bar w)^{-1}Y_0(w)}{\pi(w-\bar w)}++O((w-\lambda)\log(w-\lambda))$$
By considering a loop $\gamma$ coming from $\infty$ to $w$, making a small circle around $\lambda$, and then going back to $\infty$, we get:
$$P=\rho(\gamma)=Y_0(w)^{-1}(\Id_2+2i\pi A_i+O(w-\lambda))Y_0(w)$$
We conclude that
\begin{align*}
\Tr((P^{\pm 1}-\Id_2)R_1^\nu(w))&=\pm 2i\Tr\left(e^{i\nu}A_iB+e^{-i\nu}\frac{Y_0(w)Y_0(\bar w)^{-1}}{w-\bar w}A_i\right)\\
\Tr((P^{\pm 1}-\Id_2)R_2^\nu(w))&=O(w-\lambda)
\end{align*}
Since $P^{\pm 1}-\Id_2$ has rank 1,
$$\det(\Id_2+\rK(w,b)(P^{\pm 1}-\Id_2)\rK_\rho^{-1}(b,w))=1+\rK(w,b)\Tr((P^{\pm 1}-\Id_2)\rK_\rho^{-1}(b,w))$$
and then
$$\det(\rK'_\rho\rK_\rho^{-1})=1\pm\rK(w,b)\Re\left(2i e^{i\nu(b)}\Tr\left(e^{i\nu(w)}A_iB+e^{-i\nu(w)}\frac{Y_0(w)Y_0(\bar w)^{-1}}{w-\bar w}A_i\right)\right)+o(\delta)$$
Taking into account $\lambda'_1-\lambda_1=\pm i(b-w)$ and $\rK(w,b)=(b-w)e^{-i\nu(b)-i\nu(w)}$, we get
$$\det(\rK'_\rho\rK_\rho^{-1})=1+\Re\left(2(\lambda'_1-\lambda_1)\Tr\left(A_iB+e^{-2i\nu(w)}\frac{Y_0(w)Y_0(\bar w)^{-1}}{w-\bar w}A_i\right)\right)+o(\delta)$$
On the other hand,
$$\frac{\partial}{\partial\lambda_i}\log\tau=\Tr(A_iB), \frac{\partial}{\partial\lambda_{\bar i}}\log\tau=\Tr(\overline{A_iB})
$$
on $\Lambda$, and consequently
$$\log\tau(\lambda'_1,\lambda_2,\dots,\overline{\lambda'_1},\dots)-\log\tau(\lambda_1,\lambda_2,\dots,\overline{\lambda_1},\dots))
=2\Re((\lambda'_1-\lambda_1)\Tr(A_iB))+o(\lambda'_1-\lambda_1)
$$
Given this local estimate for a microscopic displacement of the puncture, we may integrate to estimate the variation under macroscopic displacements. Observe that $e^{-2i\nu(w)}=\pm 1$ depending on whether $w$ corresponds to a horizontal or vertical edge of the primal square lattice. By choosing the cut to run along a simple path on that primal lattice, the sign $e^{-2i\nu(w)}$ alternates along the cut. 

In conclusion we have obtained:

\begin{Lem}\label{Lem:discrvar}
Let $\lambda_1,\dots,\lambda_n$ be distinct punctures in $\H$ and $\rho:\pi_1(\H\setminus\{\lambda_1,\dots,\lambda_n\})\rightarrow\SL_2(\R)$ be a representation with unipotent local monodromies close enough to the identity.  There is a compact neighborhood of $(\lambda_1,\dots,\lambda_n)$ s.t. for $(\lambda'_1,\dots,\lambda'_n)$ in that neighborhood,
$$\log\det(\rK'_\rho\rK_{\Id}^{-1})-\log\det(\rK_\rho\rK_{\Id}^{-1})=\log\tau(\lambda'_1,\dots,\lambda'_n;\rho)-\log\tau(\lambda_1,\dots,\lambda_n;\rho)+o(1)$$
uniformly as the mesh $\delta\searrow 0$.
\end{Lem}

\subsection{Near boundary estimates}

We have estimated precisely the logarithmic variation of determinants for punctures away from the boundary. In order to complete the argument, we need to show that, for punctures close to the boundary, $\det(\rK_\rho\rK_{\Id}^{-1})$ is close to 1 uniformly in the mesh; this is the discrete counterpart of Lemma \ref{Lem:taubound}. As earlier, the argument is based on constructing a parametrix for $\rK_\rho$, but we shall need here somewhat different constructions and (less precise, but uniform near the boundary) estimates.

We start with the case $n=1$: we have a single puncture $\lambda\in\H$ (and, by reflection, $\overline{\lambda}\in (-\H)$). We can take the branch cut to be vertical. Up to conjugation, the corresponding monodromy matrix can be chosen to be $P=\left(\begin{array}{cc} 1&1\\0&1\end{array}\right)$.

The operator $\rK_\rho\rK^{-1}-\Id$ has finite rank and is strictly upper triangular (for a suitable ordering of vertices); consequently $\det(\rK_\rho\rK_{\Id}^{-1})=1$ (compare with Corollary \ref{Cor:tau1}).
Then we need to estimate the inverting kernel $\rK_\rho^{-1}$. We can construct it essentially as in Lemma \ref{Lem:1punct}, with the branch cut running from $\lambda$ to $\bar{\lambda}$ (rather than from $0$ to $\infty$) and also taking into account the boundary condition on $\R$.

If we start from 
$$b\mapsto\rK^{-1}_{\Id_2}(b,w)=(\uK^{-1}(b,w)+e^{-2i\nu(w)}\uK^{-1}(b,\bar w))\Id_2$$
we have a discrete holomorphic function (except at $w$) on $\Xi$, with an incorrect jump across the cut. If $\lambda'=\Re(\lambda)+is\Im(\lambda)$, $s\in(0,1)$, is another point on the cut, one can construct an explicit discrete holomorphic function with jump $1$ (on $\Xi\cap\Gamma$), across $[\overline{\lambda'},\lambda']$, namely 
$$b\mapsto \Log(b-\lambda')-\Log(b-\overline{\lambda'})=O(\Im(\lambda')/|b-\lambda'|)$$ 
By varying $\lambda'$ along the cut one can generate a discrete holomorphic function with prescribed jump (proceeding similarly on $\Xi\cap\Gamma^*$). This allows to construct an inverting kernel $\rK_{P,\lambda}^{-1}$ that vanishes at infinity. 

We are interested in particular in the case where $w$ is adjacent to the cut $[\bar\lambda,\lambda]$ and $b$ is at macroscopic distance of $\lambda$ (up to translation we assume that $\Re(\lambda)=O(\delta)$). In order to estimate the correction $\rK^{-1}_{P,\lambda}-\rK^{-1}_{\Id_2,\lambda}$, it is convenient to change gauge and take the cut on $[\lambda,i\infty)\cup(-i\infty,\bar\lambda]$ (so that the cut is as far away of the pole $w$ as possible). Reasoning as in Lemma \ref{Lem:taubound} we obtain (for $w$ within $O(\delta)$ of $[\bar\lambda,\lambda]$ and $\Re(b)$ of order 1)
$$|(\rK^{-1}_{P,\lambda}-\rK^{-1}_{\Id_2,\lambda})(b,w)|\leq C\sum_{k=|w-\lambda|\delta^{-1}}^\infty\frac{\delta}{(k\delta)(k\delta+1)}=O(\log|w-\lambda|)$$
We also need an estimate when $w$ is within $O(\delta)$ of $\lambda$ and $\Re(b)$ of order 1. In this case we keep the cut on $[\bar\lambda,\lambda]$ to obtain
$$|(\rK^{-1}_{P,\lambda}-\rK^{-1}_{\Id_2,\lambda})(b,w)|\leq C\sum_{k=1}^{\delta^{-1}\Im\lambda}\frac{\Im(\lambda)}{k^2\delta}=O(\delta^{-1}\Im\lambda)$$

We now turn to the general case, where $\lambda_1,\dots,\lambda_n$ are at pairwise macroscopic distance (bounded away from $0$) and close to the boundary; $\Im(\lambda_i)\leq\frac \eps 3$, $\eps$ small but fixed (as $\delta\searrow 0$). We construct a parametrix for $\rK_\rho$ in the following way. For any $w$, we set
\begin{align*}
\rS_\rho(b,w)&=\rK^{-1}_{P_i,\lambda_i}(b,w)&{\rm\ if\ }|b-\lambda_j|\geq\eps{\rm\ for\ } j\neq i\\
&=0&{\rm otherwise}
\end{align*}
where $\lambda_i$ is the puncture closest to $w$ and $P_i$ is the unipotent monodromy matrix around $\lambda_i$. Set $T=\rK_\rho\rS_\rho-\Id$, an operator on $(\R^2)^{\Xi_W}$. By construction, $T(w,w')=0$ unless $w$ is a distance $\eps$ of a puncture, in which case $T(w,w')=O(\delta\log|\lambda'-w'|)$ ($\lambda'$ the puncture closest to $w'$). Thus for fixed $w'$ at distance $\eta$ of the closest puncture
$$\sum_{w}|T(w,w')|=O(\eps\log\eta)$$
and more generally $\sum_{w}|T^n(w,w')|=O((\eps\log\eps)^{n-1}\eps\log\eta)$ (where $T^n$ denotes the $n$-th power of the kernel operator $T$).
This justifies the expansion
$$\rK_\rho^{-1}=\rS_\rho(\Id+T)^{-1}=\rS_\rho-\rS_\rho T+\rS_\rho T^2-\dots$$
Since $S_\rho(.,w)$ vanishes at infinity, one verifies easily that so does $\rK_\rho^{-1}(.,w)$.
Consider the situation where $b$ and $w$ are at small but macroscopic distance of a puncture $\lambda$, say
$$3\eps<|b-\lambda|<4\eps<5\eps<|w-\lambda|<6\eps$$
Then 
$$\rK_\rho^{-1}(.,w)=\rS_\rho(.,w)-\rS_\rho T(\Id+T)^{-1}(\delta_w\Id_2)$$
with $T(\Id+T)^{-1}(\delta_w\Id_2)=O(\eps\log\eps)$ at distance $\eps$ of punctures and $0$ otherwise. It follows that $\rK_\rho^{-1}(b,w)=O(1+|\eps\log\eps|)$.

Finally we wish to estimate $(K_\rho^{-1}-\rS_\rho)(w,b)$ for $b,w$ within $O(\delta)$ of a puncture $\lambda$. More precisely, we wish to estimate the matrix element
$$\Tr((P-\Id_2)(K_\rho^{-1}-\rS_\rho)(w,b))$$
As before, we may assume $P=\left(\begin{array}{cc}1&1\\0&1\end{array}\right)$ up to gauge change. We observe that $\rS_\rho^{-1}(w,b)\left(\begin{array}{c}1\\0\end{array}\right)$ is single-valued and equal to $\vphantom{()}^t(\uK^{-1}(b,w)\pm\uK^{-1}(b,w),0)$, and in particular is of order 1 if $w$ is near to the singularity and $b$ is at macroscopic distance.

If $B=B(\lambda,c\eps)$, we have
$$\rK_\rho^{-1}(.,w)=\ind_B\rS_\rho^{-1}+\rK_\rho^{-1}(\rK_\rho(\ind_{B^c}\rS_\rho(.,w)))$$
since both sides have the same ``residue" $\delta_w\Id_2$, are $\rho$-multivalued and vanish at infinity. By the estimate on $\rK_\rho^{-1}$ at macroscopic distance. It follows that
$$f_w(b)=(K_\rho^{-1}-\rS_\rho)(w,b)\left(\begin{array}{c}1\\0\end{array}\right)=O(1)$$
for $w$ near the singularity and for $b$ at small but macroscopic distance of the singularity. As in Lemma \ref{Lem:1punct}, it follows that
$$f_w(b)=O\left(\left(\begin{array}{c}1+\log|b-\lambda|\\1\end{array}\right)\right)$$
By considering the second component, we get
$$\Tr((P-\Id_2)(K_\rho^{-1}-\rS_\rho)(w,b))=O(1)$$
for $w,b$ close to the puncture. Upon displacing the puncture $\lambda$ across the edge $(wb)$, $\log\det(\rK_\rho{\rK_{\Id}}^{-1})$ is incremented by
$$O(\delta\Tr(P-\Id_2)(K_\rho^{-1}-\rS_\rho)(w,b)))$$
This gives the following estimate:

\begin{Lem}\label{Lem:nearbounddiscr}
Fix a representation $\rho:\pi_1(\H\setminus\Lambda)\rightarrow\SL_2(\R)$ with unipotent local monodromies and $\eta>0$. Then there is $C>0$ such that if $|\Re(\lambda_i-\lambda_j)|\geq\eta$ for $i\neq j$ and $\Im(\lambda_i)\leq\eps$ for all $i$, then
$$|\log\det(\rK_\rho\rK_{\Id}^{-1})|\leq C\eps$$
for $\eps>0$ small enough.
\end{Lem}

\section{$\SLE_4$ martingales via $\tau$-functions}

Consider a chordal $\SLE_\kappa$ in the upper half-plane $\H$; let $(g_t)_t$ be the $\SLE$ flow (extended by reflection: $g_t(\bar z)=\overline{g_t(z)}$), $\gamma_t$ the tip, $Z_t=g_t(\gamma_t)$, $\lambda_i$ a marked point in $\C$, $\Lambda^i_t=g_t(\lambda_i)$, so that:
\begin{align*}
dZ_t&=\sqrt\kappa dB_t\\
d\Lambda^i_t&=\frac2{\Lambda^i_t-Z_t}dt\\
\frac{d(\Lambda^i_t-\Lambda^j_t)}{\Lambda^i_t-\Lambda^j_t}&=-\frac{2}{(Z_t-\Lambda^i_t)(Z_t-\Lambda^j_t)}dt\\
\frac{dg'_t(\lambda_i)}{g'_t(\lambda_i)}&=-\frac 2{(\Lambda^i_t-Z_t)^2}dt
\end{align*}
Let $M_t=\prod_i (Z_t-\Lambda^i_t)^{\alpha_i}\prod_ig_t'(\lambda_i)^{\beta_i}\prod_{i<j}(\Lambda^i_t-\Lambda^j_t)^{\gamma_{ij}}$. Then:
\begin{align*}
\frac{dM_t}{M_t}=&\left(\sum_i\frac{\alpha_i}{Z_t-\Lambda^i_t}\left(\sqrt\kappa dB_t-\frac{2dt}{\Lambda^i_t-Z_t}\right)\right)+\frac\kappa 2\sum_i\frac{\alpha_i(\alpha_i-1)}{(Z_t-\Lambda^i_t)^2}dt
+\kappa \sum_{i<j}\frac{\alpha_i\alpha_j}{(Z_t-\Lambda^i_t)(Z_t-\Lambda^j_t)}dt\\
&-\sum_i\frac{2\beta_i}{(\Lambda^i_t-Z_t)^2}dt-\sum_{i<j}\frac{2\gamma_{ij}}{(Z_t-\Lambda^i_t)(Z_t-\Lambda^j_t)}dt
\end{align*}
so that $M$ is a (local) martingale provided: 
$$\beta_i=\frac\kappa 4\alpha^2_i+(1-\frac\kappa 4)\alpha_i,{\rm\ }\gamma_{ij}=\frac\kappa2\alpha_i\alpha_j$$
This may be specialized to the case: $\kappa=4$, $\lambda_{2j+1}\in\H$, $\lambda_{2j+2}=\overline{\lambda_{2j+1}}$, $\alpha_{2j+2}=-\alpha_{2j+1}$, in which case $M$ is a natural regularization of the electric correlator $\langle :\prod\exp(i\sum_j\alpha_j\phi(\lambda_{2j+1})):\rangle$, where $\phi$ is the (properly normalized) corresponding free field (see eg the discussion in \cite{Dub_tors}).

Let us now consider an $\SL_2$ (rather than $\GL_1$) version of this. As before we consider a representation $\rho:\pi_1(\H\setminus\{\lambda_1,\dots,\lambda_n\})\rightarrow \SL_2(\C)$ (with no additional condition for now) and $\tau$ the associated $\tau$-function (well defined up to multiplicative constant locally on the configuration space).

\begin{Lem}\label{Lem:locmart}
Let $\lambda_1,\dots,\lambda_n$ be punctures in $\C\setminus\{0\}$, $\rho:\pi_1(\C\setminus\{\lambda_1,\dots,\lambda_n\})\rightarrow \SL_2(\C)$ a representation.
\begin{enumerate}
\item If $(g_t)_t$ is a chordal $\SLE_4$ in $(\H,0,\infty)$, then under \eqref{eq:Schlesinger}, \eqref{eq:Fuchstot}, \eqref{eq:dlogtau},
$$M_t=\left(\prod_j g'_t(\lambda_j)^{\Tr(A_j^2)/2}\right)\tau(\Lambda_t)Y_0(Z_t;\Lambda_t)$$ 
is a (matrix-valued) local martingale.
\item  If $(g_t)_t$ is a chordal $\SLE_4(-2)$ in $(\H,z,w,\infty)$, then under \eqref{eq:Schlesinger}, \eqref{eq:Fuchstot}, \eqref{eq:dlogtau},
$$M_t=\left(\prod_j g'_t(\lambda_j)^{\Tr(A_j^2)/2}\right)\tau(\Lambda_t)\Tr(Y_0(W_t;\Lambda_t)^{-1}Y_0(Z_t;\Lambda_t))$$ 
is a (scalar) local martingale up to $\tau=\min_i\inf\{t\geq 0:\Lambda^i_t=W_t\}$.
\end{enumerate}
\end{Lem}
The determination issues (multivaluedness of $Y_0$) are resolved by picking a version which is continuous under the $\SLE$ flow.
\begin{proof}
\begin{enumerate}
\item
For generic monodromy data, this is associated to a system \eqref{eq:Fuchstot}
$$dY(z;\lambda)=\sum_j A_jY\frac{d(z-\lambda_j)}{(z-\lambda_j)}$$
where we now regard $Y$ as a function of $z$ {\em and} the punctures $(\lambda_1,\dots)$, with normalization: $Y_0(\infty;\lambda)=\Id_2$. As is well-known, this system is completely integrable provided the $A_j$'s satisfy the Schlesinger equations; its monodromy is then locally constant. Consequently,
$$dY_0(Z_t;\Lambda_t)=\sum_j \frac{A_j}{Z_t-\Lambda^j_t}\left(\sqrt\kappa dB_t-\frac{2dt}{\Lambda^j_t-Z_t}\right)Y_0+\frac\kappa2\left(-\sum_j\frac{A_j}{(Z_t-\Lambda_t^j)^2}+\left(\sum_j\frac{A_j}{Z_t-\Lambda^j_t}\right)^2\right)Y_0dt$$
With the $A_j$'s in ${\mf sl}_2$ (traceless), we have: $A_j^2=\Tr(A_j)A_j-\det(A_j)\Id_2=-\det(A_j)\Id_2=\frac 12\Tr(A_j^2)\Id_2$, and by polarization:
$$A_jA_i+A_iA_j=\Tr(A_iA_j)\Id_2$$
Hence if $\kappa=4$, we have
$$dY_0(Z_t;\Lambda_t)=\left(\sum_j \frac{A_j}{Z_t-\Lambda^j_t}\right)Y_0\sqrt\kappa dB_t+2\left(\sum_j\frac{\Tr(A^2_j)}{2(Z_t-\Lambda_t^j)^2}+\left(\sum_{i<j}\frac{\Tr(A_iA_j)}{(Z_t-\Lambda^i_t)(Z_t-\Lambda^j_t)}\right)\right)Y_0dt$$
Let $\tau$ be the associated $\tau$-function, given by \ref{eq:dlogtau}, so that
$$d\log\tau(\Lambda_t)=-\sum_{i\neq j}\frac{\Tr(A_iA_j)}{(Z_t-\Lambda^i_t)(Z_t-\Lambda^j_t)}dt$$
Consequently, if
$$M_t=\left(\prod_j g'_t(\lambda_j)^{\Tr(A_j^2)/2}\right)\tau(\Lambda_t)Y_0(Z_t;\Lambda_t)$$
then $M$ is a (matrix-valued) local martingale. 

\item To make covariance clearer, we work with a chordal $\SLE_4(-2)$ with seed at $z$ and force point at $w$, $z\neq w\in\R$. Up to time change and first hitting time of $w$, this is a chordal $\SLE_4$ from $z$ to $w$.

If $W_t=g_t(w)$, then 
\begin{align*}
dZ_t&=2dB_t+\frac{2}{W_t-Z_t}dt\\
dW_t&=\frac{2}{W_t-Z_t}dt\\
dY_0(W_t;\Lambda_t)^{-1}&=Y_0(W_t;\Lambda_t)^{-1}\sum_j A_j \frac{2dt}{(W_t-Z_t)(\Lambda^j_t-Z_t)}
\end{align*}
Consider $N_t=Y_0^{-1}(W_t;\Lambda_t)Y_0(Z_t;\Lambda_t)$. Then
\begin{align*}
dN_t&=Y_0(W_t)^{-1}\left(
\sum_j \frac{A_j}{Z_t-\Lambda^j_t}\left(\sqrt\kappa dB_t+\frac{2dt}{W_t-Z_t}-\frac{2dt}{\Lambda^j_t-Z_t}\right)+\sum_j A_j \frac{2dt}{(W_t-Z_t)(\Lambda^j_t-Z_t)}\right)Y_0(Z_t)\\
&\hphantom{=}+Y_0(W_t)^{-1}\frac\kappa 2\left(-\sum_j\frac{A_j}{(Z_t-\Lambda_t^j)^2}+\left(\sum_j\frac{A_j}{Z_t-\Lambda^j_t}\right)^2\right)Y_0(Z_t)dt\\
&=Y_0(W_t)^{-1}\left(\sum_j \frac{A_j}{Z_t-\Lambda^j_t}2 dB_t+2
\left(\sum_j\frac{A_j}{Z_t-\Lambda^j_t}\right)^2dt\right)Y_0(Z_t)
\end{align*}
With the same cancellations as above, we see that:
$$M_t=\left(\prod_j g'_t(\lambda_j)^{\Tr(A_j^2)/2}\right)\tau(\Lambda_t)Y_0(W_t;\Lambda_t)^{-1}Y_0(Z_t;\Lambda_t)$$
is a local martingale up to the first hitting time of $z=w$. By Lemma \ref{Lem:Ito}, in order for $M$ to be a local martingale up to $\tau$, we need the additional condition:
$$\partial_w\Tr(Y_0(w)^{-1}Y_0(z))=\partial_z\Tr(Y_0(w)^{-1}Y_0(z))=0$$
at $z=w$. This follows immediately from
$$\partial_zY_0(z)=A(z)Y_0(z)$$
with $A$ traceless. (Remark that at time $\tau$, $Y_0(Z_t)\neq Y_0(W_t)$ in general, since they differ by the monodromy of the loop created by the trace, which may now encircle some punctures).
\end{enumerate}
\end{proof}

If the $\lambda_j$'s are paired by reflection, this should correspond to double-dimer $\SL_2(\C)$ observables with a chordal path from $z$ to $w$, with the following natural interpretations: $\tau$ the partition function with punctures and no chordal path, $1/(z-w)$ the partition function with a chord and no punctures, and $\tau S(z,w)$ the partition function with punctures and a chord.

\begin{Lem}\label{Lem:UI}
With the notations of Lemma \ref{Lem:locmart}, if furthermore $\rho$ is close enough to the trivial representation and with unipotent local monodromies, then $\sup_t|M_t|$ is integrable, and {\em a fortiori} $M$ is a uniformly integrable martingale.
\end{Lem}
\begin{proof}
By Lemma \ref{Lem:tauboundunif}, by taking $\rho$ close enough to the identity we have $t\mapsto|\log\tau(\Lambda_t;\rho)|$ bounded (for all times, by a deterministic constant).

For $Y_0(Z_t;\Lambda_t)$, the main difficulty is that is evaluated along the $\SLE$ trace $\gamma$, which may itself wind many times around the punctures. Draw crosscuts from the punctures $\Lambda_0$ to $\R$ at pairwise positive distance and set $N$ to be the number of times the trace $\gamma$ travels between two distinct crosscuts. Divide the crosscuts in $O(n/q)$ segments of diameter $O(q/n)$ ($q$ a large integer). By the pigeonhole principle, if $N\geq n$, one of these contains $q$ macroscopic, disjoint segments of the trace. By the annular crossing estimate of \cite{Werness_cross} (Theorem 5.7) - and the trivial union bound - this gives:
$$\P(N\geq n)\leq c_1\frac{n}q\left(c_2q/n\right)^{\beta q}$$
for some positive constants $\beta,c_1,c_2$. By taking $q=\lfloor\eps n\rfloor$, $0<\eps<c_2^{-1}$, this shows that $N$ has exponential tail.

In the set-up of Lemma \ref{Lem:locmart} (2), one observes that the $\SLE_4(-2)$ trace is contained in a $\CLE_4$ and one can use the annular crossing estimate of Lemma \ref{Lem:CLEarm} instead.

In order to evaluate $Y_0(Z_t;\Lambda_t)$ (which is continuous in $t$), we progressively deform the original crosscuts (so that they stay disjoint of the trace and each other). When replacing these crosscuts by ``shorter" ones (that do not wind around each other), we have to conjugate the $A_i$'s by a matrix $P$ s.t. $P=O((1+\eps)^N)$, where $\eps>0$ is arbitrarily small provided that $\rho$ is close enough to the trivial representation (see Lemma \ref{Lem:tauboundunif} for a similar argument).

Then we consider 
$Y_0(z,\Lambda)$ obtained by solving \eqref{eq:Fuchs} along the boundary, which leads to an estimate $Y_0(Z_t;\Lambda_t)=O((1+\eps')^N)$, with $\eps'$ arbitrarily small for $\rho$ close enough to the trivial representation. Since $N$ itself has exponential tail, this gives the desired result.
\end{proof}

\begin{Lem}\label{Lem:oneloop}
Consider the punctured half-plane $\H\setminus\{\lambda_1,\dots,\lambda_n\}$ and $\rho:\pi_1(\H\setminus\{\lambda_1,\dots,\lambda_n\})\rightarrow\SL_2(\R)$ a representation with unipotent local monodromies close enough to the trivial representation. 
\begin{enumerate}
\item
Let $\gamma$ be a chordal $\SLE_4$ trace from $0$ to $\infty$ in $\H$, $\H^+$ and $\H^-$ the connected components of $\H\setminus\gamma$ (resp. to the right and to the left of $\gamma$). Then
$$Y_0(0)\tau(\H\setminus\Lambda,\rho)=\E(\rho(\partial\H^+)\tau(\H^+\setminus\Lambda,\rho_+)\tau(\H^-\setminus\Lambda,\rho_-))$$
\item Let $\gamma$ be a chordal $\SLE_4(-2)$ started at $(0,0^+)$ and stopped at the first time $\tau$ when it completes a non-trivial loop $\delta$ w.r.t. the punctures. Then
$$\tau(\H\setminus\Lambda,\rho)=\E\left(\frac{\Tr(\rho(\delta))}2\tau(\H\setminus\gamma_\tau\setminus\Lambda,\rho)\right)$$
\end{enumerate}
\end{Lem}
Here $\partial\H^+$ is the clockwise oriented boundary of $\H^+$ and $Y_0(0)$ is obtained by solving \eqref{eq:Fuchs} from $\infty$ to $0$ along the positive half-line. 
\begin{proof}
\begin{enumerate}
\item
From Lemmas \ref{Lem:locmart}, \ref{Lem:UI}, we have 
$$M_0=\E(\lim_{t\rightarrow\infty}M_t)$$
From Lemma \ref{Lem:pinch} we have 
$$\lim_{t\rightarrow\infty}\tau(\H\setminus\gamma_{[0,t]}\setminus\Lambda,\rho)=\tau(\H^+,\rho_+)\tau(\H^-,\rho_-)$$
We are left with evaluating $\lim Y_0(Z_t;\Lambda_t)$. As in Lemma \ref{Lem:pinch} we assume up to relabelling that $\lambda_1,\dots,\lambda_k$ (resp. $\lambda_{k+1},\dots,\lambda_n$) are in $\H^-$ (resp. $\H^+$). When $t$ is large, up to homography we may assume $\Re(\Lambda^i_t)=O(1)$ for $i\leq k$, $\Re(\Lambda^i_t-M_t)=O(1)$ for $1\leq i\leq n$ with $M_t\rightarrow\infty$; and $1\ll z_t\ll M_t\ll w_t$ where $z_t,w_t$ are the endpoints of the $\SLE$ in these coordinates (this follows eg from simple harmonic measure considerations). In this set-up, the variation of $Y_0$ along the imaginary half-line $[z_t,z_t+i\infty)$ and the real half-line $[w_t,\infty)$  is $o(1)$ and consequently $Y_0(Z_t,\Lambda_t)=\rho(\partial\H^+)+o(1)$.

\item The only modification needed is to justify
$$\Tr(\rho(\delta))=\lim_{t\nearrow\tau}\Tr(Y_0(W_t;\Lambda_t)^{-1}Y_0(Z_t;\Lambda_t))$$
If $\sigma$ is the last time (strictly) before $\tau$ s.t. $W_\sigma=Z_\sigma$, then $Y_0(W_\sigma)=Y_0(Z_\sigma)$. Moreover, for $t$ close to $\tau^{-}$, $Y_0(W_t;\Lambda_t)^{-1}Y_0(Z_t;\Lambda_t)$ is close to the monodromy of the loop $\delta$ in $\H\setminus\gamma_{[0,t]}$, which is (up to conjugation) $\rho(\delta)$.
\end{enumerate}
\end{proof}

\begin{Thm}\label{Thm:CLEtau}
Consider a (nested) $\CLE_4$ $(\ell_\alpha)_{\alpha\in A}$ in $\H$ and punctures $\lambda_1,\dots,\lambda_n$. If $\rho:\pi_1(\H\setminus\{\lambda_1,\dots,\lambda_n\})\rightarrow\SL_2(\R)$ is a representation with unipotent local monodromies close enough to the trivial representation, then
$$\tau(\H\setminus\Lambda,\rho)=\E\left(\prod_{\alpha\in A}\frac{\Tr(\rho(\ell_\alpha))}2\right)$$
\end{Thm}
\begin{proof}
The variable 
$$T_\rho=\prod_{\alpha\in A}\frac{\Tr(\rho(\ell_\alpha))}2$$
contains a.s. finitely many non unit factors - since local monodromies are unipotent, only loops encircling at least two punctures contribute. If we map $\H$ eg to the unit disk $\U$ and draw disjoint crosscuts from the punctures to the boundary, if $N(\ell)$ denotes the number of successive visits of distinct crosscuts by a loop $\ell$, then $|T_\rho|\leq (1+\eps)^{\sum_\alpha N(\ell_\alpha)}$, with $\eps$ small when $\rho$ close to the trivial representation. By Lemma \ref{Lem:CLEarm}, $\sum_\alpha N(\ell_\alpha)$ has exponential moments and consequently for a given $p\geq 1$ and punctures $\lambda_i$, $T_\rho$ is in $L^p$ for $\rho$ close enough to the trivial representation. Note that the tail estimate is uniform for $\CLE$s in subdomains of $\U$ (restricting the crosscuts and representations).

We use the exploration description of $\CLE_4$ introduced in \cite{Sheff_explor} (the invariance - in law - w.r.t. choices is established in \cite{SheWer_CLE}). 

We consider first a discrete skeleton of the exploration. In the upper half-plane model, we start from $\H\setminus\{\lambda_1,\dots,\lambda_n\}$ with $\lambda=1$ and seed the exploration at $z=0$. We run an $\SLE_4(-2)$ started from $(0,0^+)$ up to $\tau_1$, the first time the trace completes a non-trivial loop w.r.t the punctures. Let $\gamma_1$ be the trace at time $\tau_1$. Let $i(1)$ be the smallest index in $\{1,\dots,n\}$ s.t. $\lambda_i$ is not the single puncture in its connected component of $\H\setminus\gamma_1$. Then run an $\SLE_4(-2)$ (seeded eg at the tip of $\gamma_1$) in that connected component; this gives a trace $\gamma_2$. Then iterate this procedure for $K$ steps, where $K$ is the smallest integer s.t. each connected component of $\H\setminus\cup_{i\leq K}\gamma_i$ contains at most one puncture. Each step produces either one loop encircling at least two punctures; or the outermost loop encircling just one puncture. By the previous argument, $K$ is a.s. finite with exponential tail. 

Then we have the sequential description
$$T_\rho=\prod_{\alpha\in A}\frac{\Tr(\rho(\ell_\alpha))}2=\prod_{i=1}^K\frac{\Tr(\rho(\gamma_i^l))}2$$
where $\gamma_i^l$ is the last loop completed by the $\SLE$ trace $\gamma_i$. Let
$$M_k=\tau(\H\setminus\cup_{i\leq k}\gamma_i\setminus\Lambda;\rho)\prod_{i=1}^{k\wedge K}\frac{\Tr(\rho(\gamma_i^l))}2$$
Then by Lemma \ref{Lem:oneloop}, $M$ is a discrete-time martingale w.r.t. the filtration $(\sigma(\gamma_1,\dots,\gamma_k))_k$. We have $M_K=T_\rho$ (by Corollary \ref{Cor:tau1}) and $M$ is uniformly integrable for $\rho$ close enough to the trivial representation by the argument above and Lemma \ref{Lem:tauboundunif}. Thus $M_0=\E(M_K)$, as claimed.
\end{proof}

\section{Technical results}

In this section we collect several lemmas bearing on crossing estimates for $\CLE$; martingales for $\SLE_\kappa(\rho)$ in the Dirichlet process regime; {\em a priori} uniform bounds on $\tau$-functions; and the infinite volume limit of double dimers (in the upper half-plane).  

\subsection{Crossing estimates}

The following superexponential tail bound for $N_{xy}$, the number of (nested) $\CLE$ loops encircling $x\neq y$, is not logically needed in the main results (and we shall give a somewhat sketchy proof). It may be used to relax the assumption that $\rho$ be close to the trivial representation (in the case of two punctures).

\begin{Lem}
Let $x,y$ be distinct points in a simply connected domain $D$ and $N_{x,y}$ the number of loops surrounding both $x$ and $y$ in a $\CLE_\kappa$, $\kappa\leq 4$. Then the distribution of $N_{x,y}$ has superexponential tail.
\end{Lem}
\begin{proof}
Map conformally $(D,x,y)$ to $(\U,0,r)$ $0<r<1$. We consider the exploration of the $\CLE_\kappa$ in radial parameterization (fixing $0$). In the radial Loewner equation, we have
$$dr_t\geq r_t\frac{1-r_t}2$$
and then 
$$\log\frac{r_t}{1-r_t}\geq \log\frac{r_0}{1-r_0}+\frac t2$$
and $1-r_t\leq c_0e^{-t/2}$, $c_0$ a constant depending on $r_0$. On the one hand, the probability that the first loop around $0$ does not disconnect $0$ from $r$ is $O((1-r)^\alpha)$ for some $\alpha>0$ (eg by Schramm's formula \cite{Sch_percform}). On the other hand, the probability that the first loop is completed before time $t$ is $O(e^{-c/t})$, eg from \cite{SSW_rad}; and the probability that $n$ loops are completed before $t$ is $O(e^{-cn^2/t})$.

Let $\tau_n$ be the time at which the $n$-th loop around $0$ is completed. Then on the event $\{N_{x,y}\geq 2n\}$, either the first $n$ loops are completed at time $\tau_n$ with $\tau_n\leq\sqrt n$; or $(1-r_{\tau_n})\geq e^{-c\sqrt n}$. This gives the (crude) estimate
$$\P(N_{xy}\geq 2n)\leq e^{-cn^{3/2}}$$
\end{proof}

\begin{Lem}\label{Lem:CLEarm}
If $D\subset\C$ is a simply connected domain of diameter $\leq C$, $\kappa\in (\frac 83,4]$, then there are $c_1,c_2,\beta>0$ (depending on $C,\kappa$) s.t. for all $x\in\C$, $0<r<R$ and $k\in\N$, 
$$\P({\rm a\ }\CLE_\kappa{\rm\ in\ }D{\rm\ contains\ }k{\rm\ distinct\ crossings\ of\ }\{z:r<|z-x|<R\})\leq c_1(c_2r/R)^{\beta k}$$
\end{Lem}
The assumption on the diameter bound is for convenience and could be dispensed with.
\begin{proof}
In the case where $k=1$ and the annulus is in $D$, this follows easily from \cite{SSW_rad}, see also Lemma 3.4 in \cite{MWW_extr}. Then we can use the description of $\CLE_\kappa$ in terms of loop-soup clusters, as in \cite{SheWer_CLE}. Denote by $A(r_1,r_2)$ the annulus $\{z:r_1<|z-x|<r_2\}$. If the $\CLE$ has $k$ distinct crossings of $A(r,R)$, there are $\lfloor k/2\rfloor$ chains of loops crossing $A(r^{2/3}R^{1/3},r^{1/3}R^{2/3})$. Any loop contributing to $\ell\geq 2$ crossings must itself make $\ell$ disjoint crossings of $A(r^{1/3}R^{2/3},R)$ or $A(r,r^{2/3}R^{1/3})$; a Beurling estimate shows that the mass of such loops is $O((r/R)^{\beta_0\ell})$ for some $\beta_0>0$. So up to an event of probability $O((r/R)^{\beta_1 k})$ for some $\beta_1>0$, we have $\lfloor k/2\rfloor$ crossings  of $A(r^{2/3}R^{1/3},r^{1/3}R^{2/3})$ by chains involving disjoint sets of loops. 

The existence of such a crossing is an increasing event and its probability is a monotone function of the domain $D$ (by the restriction property of the loop-soup). Then, as in Lemma 9.6 of \cite{SheWer_CLE}, we can use the disjoint occurrence inequality for Poisson Point Processes of \cite{vdB_PPP} to get the desired bound. 
\end{proof}

\subsection{$\SLE_\kappa(\rho)$ for small $\rho$}\label{Sec:kapparho}

Here we consider the upper half-plane $\H$ with a marked seed $w$, force point $z=z_0$ and spectator points $z_1,\dots,z_n$ on the boundary (as usual spectator points in the bulk can be treated by reflection). Let $(g_t)_t$ designate a Loewner chain and set $Z_t=Z_t^0=g_t(z)$, $Z_t^i=g_t(z_i)$.

Consider the dynamics specified by
\begin{align*}
dg_t(u)&=\frac 2{g_t(u)-W_t}\\
dW_t&=\sqrt\kappa dB_t+\frac{\rho}{W_t-Z_t}dt
\end{align*}
where $B_t$ is a standard linear BM; this is non-problematic as long as $Z_t\neq W_t$. Setting $\sqrt\kappa V_t=Z_t-W_t$, we have
$$dV_t=-dB_t+\frac{\delta-1}{2V_t}dt$$
with $\delta=1+2\frac{\rho+2}\kappa$. This identifies $V$ with a Bessel process of dimension $\delta$. We consider here the case where $V_t\geq 0$ at all times.

If $\delta\geq 2$, the Bessel process $V$ does not hit $0$ (except possibly at its starting point. If $\delta>1$, $V$ can be still be defined as a nonnegative semimartingale  and $\int_0^.ds/V_s$ is a finite variation process, and we can start from $V$ and define $Z_t-Z_0=\int_0^t\frac{2}{\sqrt\kappa V_s}$, $W_t=Z_t-\sqrt\kappa V_t$. That is the original construction of \cite{LSW_restr}.

If $0\leq\delta\leq 1$, $\int_0^tds/V_s$ is no longer of finite variation (or even defined). As in \cite{Dub_exc,Sheff_explor}, this may be remediated by taking principal values of Bessel processes \cite{Ber_comp}. Specifically, there is a bicontinuous local time process $(\ell_t^x)_{t,x\geq 0}$ s.t.
$$\int_0^tf(V_s)ds=\int_0^\infty f(x)\ell_t^xx^{\delta-1}dx$$
for, say, $f$ bounded and continuous. This local time is H\"older in $x$, and then we may define 
$$p.v.\int_0^t\frac{ds}{V_s}=\int_0^\infty(\ell_t^x-\ell_t^0)x^{\delta-2}dx$$
On the set $\{t:V_t\neq 0\}$, this is differentiable with derivative $V_t^{-1}$. 
The decomposition of $V$ as the sum of a local martingale and a zero energy process reads
$$V_t=-B_t+\frac{\delta-1}{2}
p.v.\int_0^t\frac{ds}{V_s}$$
if $0<\delta<1$ and (classically)
$$V_t=B_t+\frac 12\ell_t^0$$
if $\delta=1$.

Then one defines $Z_t-Z_0=p.v.\int_0^t\frac{2}{\sqrt\kappa V_s}$, $W_t=Z_t-\sqrt\kappa V_t$. A difficulty is that $p.v.\int_0^t ds/V_s$ is not of finite variation, but of finite $p$-variation for some $p<2$. Consequently $W$ is not a semimartingale but a Dirichlet process in the sense of eg \cite{Fuku_Dirforms} (the sum of a continuous local martingale and a process with zero quadratic variation).

\begin{Lem}\label{Lem:Ito}
In the case $\delta\in (0,1]$, let $\phi=\phi(w,z,z_1,\dots,z_n)$ be a $C^2$ function s.t. 
$$\frac\kappa 2\phi_{ww}+\frac\rho{w-z}\phi_w+\sum_{i=0}^n\frac{2}{z_i-w}\phi_{z_i}-\sum_{i=0}^n\frac{2h_i}{(z_i-w)^2}\phi=0$$
and 
\begin{align*}
-\rho\phi_w+2\phi_z&=0&&{\rm\ if\ }0<\delta<1\\
\phi_w=\phi_z&=0&&{\rm\ if\ }\delta=1
\end{align*}
at $z=w$. Then
$$t\mapsto \phi(W_t,Z_t,Z^1_t,\dots,Z^n_t)\prod_{i=0}^n g'_t(z_i)^{h_i}$$
is a local martingale up to
$$\tau=\sup\{t>0: Z^i_s\neq W_s\ \forall s\in [0,t], i=1,\dots,n\}$$
\end{Lem}
This is, of course, classical in the case $\delta>1$ (where the boundary conditions at $z=w$ are not needed). The condition $\phi_z=\phi_w=0$ at $z=w$ means that $z\mapsto \phi(z,z,z_1,\dots)$ is locally constant, a natural condition in terms of $\CLE$.
\begin{proof}
Recall that $d\log g'_t(z_i)=-\frac{2dt}{(g_t(z_i)-W_t)^2}$; this accounts for the $\prod g'_t(z_i)^{h_i}$ factor. For notational simplicity we assume now that $h_0=\cdots=h_n=0$.

Here $W$ is a Dirichlet process and the $Z$'s are zero energy processes. Then eg by Section 3 of \cite{Ber_avar} we may write
\begin{align*}
\phi(W_t,Z_t,Z_t^1,\dots)-\phi(W_0,Z_0,Z_0^1,\dots)&=\int_0^t\phi'_w(W_s,\dots)dW_s+\frac\kappa 2\int_0^t\phi''_u(W_s,\dots)ds+\sum_{i=0}^n\int_0^t\phi'_{z_i}(W_s,\dots)dZ^i_s\\
&=\int_0^t\phi'_w(W_s,\dots)\sqrt\kappa dB_s+\int_0^t(-\frac{\rho}2\phi'_w+\phi'_z)dZ_s\\
&+\frac\kappa 2\int_0^t\phi''_u(W_s,\dots)ds+\sum_{i=1}^n\int_0^t\phi'_{z_i}(W_s,\dots)dZ^i_s
\end{align*}
if $0<\delta<1$. Here $\int(\dots)dB_s$ is a local martingale, the other integrals are Lebesgue-Stieltjes integrals except for
$$\int_0^t(-\frac\rho 2\phi'_w+\phi'_z)dZ_s$$
which is taken in the sense of F\"ollmer \cite{Foll_calcul}, viz. simply as the limit of Riemann sums along a deterministic sequence of subdivisions. We check easily that under %
the oblique condition $\rho\phi'_w-2\phi'_z=0$ at $z=w$, this is actually simply the converging integral
$$\int_0^t(\rho\phi'_w-2\phi'_z)(W_s,Z_s,\dots)\frac{ds}{W_s-Z_s}$$
Then terms cancel out except for the local martingale $\int_0^t\phi'_wdB_s$.

In the case $\delta=1$, there is an additional term $\int_0^t\phi'_w(W_s,\dots)d\ell^0_t$, which vanishes identically under the condition $\phi'_w=0$ at $z=w$.
\end{proof}

\subsection{Bounds on $\tau$-functions}

We have need for a priori uniform bounds on $\tau$-functions, which depend on the position of punctures (up to conformal automorphisms) and a choice of representation of the fundamental group of the punctured domain. Under exploration, the punctures stay fixed and the domain shrinks.

Let us start with a simply connected domain $D$ with punctures $z_1,\dots,z_n$. The number $n$ is fixed (or bounded). We consider disjoint cuts $\delta_1,\dots,\delta_n$ connecting $z_1,\dots,z_n$ to the boundary of $D$. Let $B_1,\dots,B_n$ be independent Brownian motions started from $z_1,\dots,z_n$ and stopped upon exiting $D$. We say that the cuts $\delta_1,\dots,\delta_n$ are {\em $\eps$-separated} if, with probability at least $\eps$, $B_i$ exits the domain $D$ without hitting $B_j$ or $\delta_j$ for $j\neq i$. 

Two advantages of this notion are conformal invariance and monotonicity (in the domain). Clearly, if $D'\subset D$ and $\delta'_i$ is the connected component of $z_i$ in $\delta_i\cap D'$, then $\delta_1,\dots,\delta_n$ $\eps$-separated in $D$ implies $\delta'_1,\dots,\delta'_n$ $\eps$-separated in $D'$.

Given a collection of cuts $\delta_1,\dots,\delta_n$ in a punctured domain $D\setminus\{z_1,\dots,z_n\}$ and a base point $z\in D$, we can consider simple loops $\gamma_1,\dots,\gamma_n$ s.t. $\gamma_i$ circles counterclockwise around $z_i$ without intersecting $\delta_j$, $j\neq i$. Then we can define a notion of size of a representation $\rho:\pi_1(D\setminus\{z_1,\dots,z_n\})\rightarrow\SL_2(\R)$ by setting  
$$|\rho|=\min_i\|\log\rho(\gamma_i)\|$$
where $\log\rho(\gamma_i)$ is unambiguous since we are considering the case of unipotent local monodromies; $\|.\|$ denotes an operator norm on matrices. This notion depends on the choice of loops $\gamma_i$ (or of cuts $\delta_i$ and base point $z$). 

\begin{Lem}\label{Lem:tauboundunif}
For any $\eps>0$, there is $\delta>0$ small and $M>0$ large s.t. for any punctured domain $D\setminus\{z_1,\dots,z_n\}$ with $\eps$-separated cuts $\delta_1,\dots,\delta_n$ and representation $\rho:\pi_1(D\setminus\{z_1,\dots,z_n\})\rightarrow\SL_2(\R)$ with $|\rho|\leq\delta$,
$$|\log(\tau(D\setminus\{z_1,\dots,z_n\},\rho))|\leq M$$
\end{Lem}
\begin{proof}
We proceed in three steps. Firstly, we uniformize and construct ``good" cuts $\tilde\delta_i$ in $\H$ with control on the length and pairwise distance. These are not necessarily deformations of the (images of) the original cuts. See Figure \ref{Fig:tauunif}. In the second step, we control the size of $\rho$ w.r.t. to the new cuts. Lastly, we obtain a bound on $\log\tau$ from the conditions on cuts and size of $\rho$.
\begin{figure}
\begin{center}\includegraphics[scale=.8]{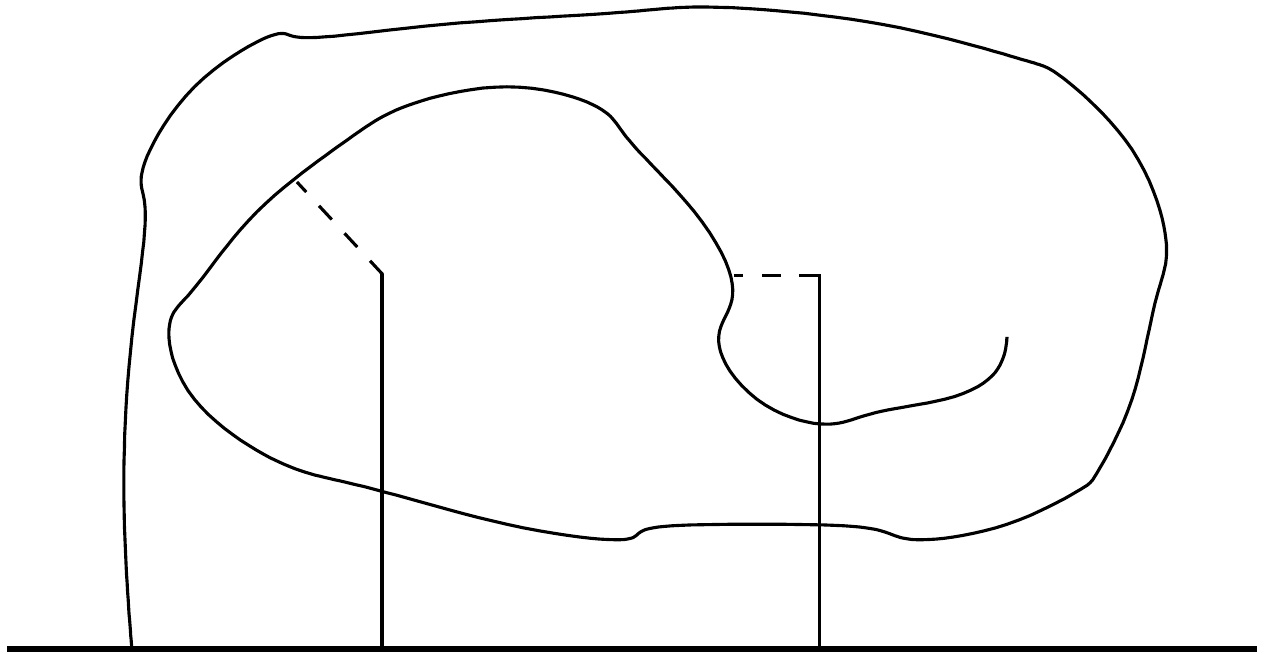}\end{center}
\caption{Original branch cuts in $\H$ (vertical). Dashed: new branch cuts in the domain cut along the curved exploration process.
}\label{Fig:tauunif}
\end{figure}

{\bf First step.} Let $\phi:(D,z_1,\dots,z_n)\rightarrow (\H,\tilde z_1,\dots,\tilde z_n)$ be a conformal equivalence. For instance from the Beurling estimate, it is clear that the pairwise hyperbolic distances between the $\tilde z_i$'s are bounded below by $\eps'=\eps'(\eps)>0$. We may for instance tile $\H$ by squares of side $\frac{\eps'}{10}  
\Im(z_1)$ and consider a path from $z_1$ to $\R$ on this square tiling with length $\leq C\Im(z_1)$ ($C$ depends on $n$) and at distance at least $\frac{\eps'}{10}  
\Im(z_1)$ from other punctures. One proceeds similarly to construct a cut from $\tilde z_2$ to $\R$, etc., ordering the punctures by decreasing imaginary parts. In this fashion we obtain piecewise linear cuts $\tilde\delta_i$ s.t. 
\begin{align*}
{\rm length}(\tilde\delta_i)&\leq C\Im(z_i)\\
\dist(\tilde\delta_i,\tilde\delta_j)&\geq \eps'\max(\Im(z_i),\Im(z_j))
\end{align*}
with $C<\infty$ and $\eps'>0$ depending on $n,\eps$. By scaling we may assume $\max\Im(z_i)=1$. Moreover there is $\eps''=\eps''(\eps')>0$ s.t. the probability that a BM started from any point on $\tilde\delta_i$ exits $\H$ before hitting a $\tilde\delta_j$, $j\neq i$, is at least $\eps''$.

{\bf Second step.} We initially have upper bounds on the $\|N_i\|$ where $\Id_2+N_i=\rho(\gamma_i)$, and the $\gamma_i$'s are loops in $D$ relative to the cuts $\delta_i$. In $\H$, we have new cuts $\tilde\delta_i$ and corresponding loops (given a base point) $\tilde\gamma_i$ and we want to bound the $\|M_i\|$'s, where $\Id_2+M_i=\rho(\tilde\gamma_i)$'s ($\pi_1(D\setminus\{z_1,\dots,z_n\})$ and $\pi_1(\H\setminus\{\tilde z_1,\dots,\tilde z_n\})$ are identified via $\phi$). 

Let $\delta'_i=\phi(\delta_i)$. We need to estimate the number of intersections of a loop equivalent to $\tilde\gamma_i$ with the cuts $\delta'_1,\dots\delta'_n$. For simplicity we choose the base point in $\H$ to be at the root of $\tilde\delta_i$ (the general case is a mild modification). We take $\tilde\gamma_i$ to run along $\tilde\delta_i$, make a small loop around $\tilde z_i$, and return to the root along $\tilde\delta_i$. In such a way the number of crossings of $\tilde\gamma_i$ with the $\delta'$'s can be chosen to be comparable with the number of crossings of $\tilde\delta_i$ by $\delta'_1,\dots,\delta'_n$. 

By the separation condition, with probability at least $\eps$, the $\delta'$'s are retracts of independent Brownian motions started at the punctures $\tilde z_i$. The number of crossings (successive visits to distinct cuts) of the $\tilde\delta_i$'s by a BM is dominated by $1+G$, $G$ a geometric variable with parameter $\eps''$. This gives a finite $N=N(\eps)$ s.t. the total number of crossings of the $\delta'_j$'s by a $\tilde\gamma_i$ is bounded by $N$. Then we may write
$$\Id_2+M_i=\prod_{j=1}^{n_i}(\Id_2\pm N_{k_j})$$
for some signs $\pm$ and some indices $k_1,\dots,k_{n_i}$, with $n_i\leq 2N$. This gives
$$M_i=\sum_j \pm N_{k_j}+\sum_{j_1<j_2}\pm N_{k_{j_1}}N_{k_{j_2}}+\cdots$$
and by submultiplicativity of operator norms
$$\|M_i\|\leq \delta'=(1+\delta)^{2N}-1$$
so that for a fixed $\eps>0$, $\delta'\searrow 0$ as $\delta\searrow 0$.

{\bf Conclusion.}

We can solve first the Schlesinger equations \eqref{eq:Schlesinger} and then \eqref{eq:dlogtau} by moving in turn each puncture from the root of the cut $\tilde\delta_i$ to its extremity. Then by a Gr\"onwall-Bihari differential inequality (Lemma \ref{Lem:Bihari}), one obtains a uniform bound on the $A$'s for punctures on the $\tilde\delta_i$'s for small enough initial data. The initial data consists in the $M_i$'s, which can be made arbitrarily small by choosing $\delta$ small enough. Then solving \ref{eq:dlogtau} (again along the punctures) gives the desired bound on $|\log\tau|$.  
\end{proof}

For the reader's convenience we provide a simple case of the Gr\"onwall-Bellman-Bihari integral inequality \cite{Bihari}.
\begin{Lem}\label{Lem:Bihari}
Let $u:[0,\infty)\rightarrow [0,\infty)$ be continuous and s.t. for $t\geq 0$
$$u(t)\leq u_0+\alpha\int_0^t u(s)^2ds$$
Then
$$u(t)\leq \frac{u_0}{1-\alpha u_0t}$$
for $t\leq (\alpha u_0)^{-1}$.
\end{Lem}
\begin{proof}
Let $v(t)=u_0+\alpha\int_0^t u(s)^2ds$, so that $u\leq v$ and $v$ is $C^1$. Then
$$v'=\alpha u^2\leq \alpha v^2$$
and it follows easily that 
$$v(t)\leq \frac{u_0}{1-\alpha u_0t}$$
\end{proof}

\subsection{Double dimers in infinite volume}\label{Sec:DDHP}

For technical reasons it is convenient to work in infinite volume, specifically in the upper half-plane. For definiteness, consider the square lattice $\Z^2$. From \cite{Ken_domino_conformal}, we know that there is an increasing sequence $\Gamma_n$ of subgraphs of $\Z^2$ exhausting the upper half-plane $\Z\times\N$ s.t. the uniform measure on perfect matchings of $\Gamma_n$ converges weakly to a probability measure on perfect matchings on $\Xi=\Z\times\N$ (each perfect matching is seen as a subgraph of $\Z\times\N$ and one consider the product $\sigma$-algebra generated by the states of individual edges). For example, one can take rectangles with even sides or rectangles with odd sides, minus a corner. Moreover, if $\rK_n^{-1}$ is the inverse Kasteleyn matrix of $\Gamma_n$, 
$$\rK_\Xi^{-1}(b,w)=\lim_{n\rightarrow\infty}\rK_n^{-1}(b,w)$$
exists for all $w$ white (resp. $b$ black) vertex in $\Xi$; $\rK_\Xi^{-1}$ is invariant under (horizontal) translations and $\lim_{b\rightarrow\infty}\rK_\Xi^{-1}(b,w)=0$ for any fixed $w$. This uniquely characterizes $\rK_\Xi^{-1}$, which, as pointed out earlier, can also be written in terms of the full plane kernel $\rK^{-1}$ by reflection arguments. 

Let us fix punctures (faces of $\Z\times\N$), branch cuts, and a representation $\rho$. We take $n$ large enough so that $\Gamma_n$ contains these punctures and cuts. For such $n$, we have from \eqref{eq:KenDD}
$$\E_{{\rm dimer}^2}^{\Gamma_n}\left(\prod_{\ell{\rm\ loop\ in\ }{\mf m}_1\cup{\mf m}_2}\frac{\Tr(\rho(\ell))}2\right)=\det((\rK_n\oplus\rK_n)_\rho(\rK_n^{-1}\oplus\rK_n^{-1}))$$
We want to take an infinite volume limit of this expression. Notice that the RHS is a (fixed) finite rank perturbation of the identity, and consequently
$$\lim_{n\rightarrow\infty}\det((\rK_n\oplus\rK_n)_\rho(\rK_n^{-1}\oplus\rK_n^{-1}))=\det(\rK\oplus\rK)_\rho(\rK_\Xi^{-1}\oplus\rK_\Xi^{-1}))$$
where the RHS has a finite Fredholm expansion (or is the determinant of any finite square block corresponding to a finite set of white vertices containing all those adjacent to the branch cuts).

In order to justify the infinite volume identity
$$\E_{{\rm dimer}^2}^{\Xi}\left(\prod_{\ell{\rm\ loop\ in\ }{\mf m}_1\cup{\mf m}_2}\frac{\Tr(\rho(\ell))}2\right)=\det(\rK\oplus\rK)_\rho(\rK_\Xi^{-1}\oplus\rK_\Xi^{-1}))$$
one can argue by weak convergence that
$$\E_{{\rm dimer}^2}^{\Xi}\left(\prod_{\ell{\rm\ loop\ in\ }{\mf m}_1\cup{\mf m}_2}\frac{\Tr(\rho(\ell))}2\right)=\lim_{n\rightarrow\infty}\E_{{\rm dimer}^2}^{\Gamma_n}\left(\prod_{\ell{\rm\ loop\ in\ }{\mf m}_1\cup{\mf m}_2}\frac{\Tr(\rho(\ell))}2\right)$$
The difficulty is in showing that the infinite volume double-dimer configuration does not contain bi-infinite paths. More precisely, we want to show that the diameter of  loops crossing a fixed branch cut of length $k$ is tight as $n\rightarrow\infty$. 
By modifying a bounded number of dimers in the superposition of configurations, a double dimer loop crossing a branch cut can be transformed into a loop going through a fixed boundary edge (two double-dimer loops using a pair of parallel edges can be merged by a local operation).  

Thus let $(b_0w_0)$ be a boundary edge within $O(1)$ of 0; and $\gamma_R$ be a simple path on the dual graph at distance of order $R$ of $(b_0,w_0)$ disconnecting a bounded component of $\Z\times\N$ (containing $B(0,R/2)$) from an unbounded component (containing $B(0,R)^{c}$). We choose $\gamma_R$ so that it crosses $O(R)$ edges, each of them with the black vertex inside and the white vertex outside. Let $\ell_0$ be the double-dimer loop through $(b_0w_0)$. Trivially,
$$\P^{\Gamma_n}(\diam(\ell_0)>R)\leq\E^{\Gamma_n}(|\{(bw)\in\ell_0, (bw){\rm\ crosses\ }\gamma_R\}|)=\sum_{(bw):(bw){\rm crosses\ }\gamma_R}\P^{\Gamma_n}((bw)\in\ell_0)$$
So we wish to bound the probability that a double-dimer loop goes through a bulk edge and a boundary at distance $\asymp R$. For this, we observe that a sliding argument similar to the one introduced by Kenyon in \cite{Ken_Lap} gives
$$\P^{\Gamma_n}_{{\rm dimer}^2}((bw)\in\ell_0)=\frac{{\mc Z}_{\rm\ dimer}(\Gamma_n\setminus\{b_0,w\}){\mc Z}_{\rm\ dimer}(\Gamma_n\setminus\{b,w_0\})}{{\mc Z}_{\rm\ dimer}(\Gamma_n)^2}$$
where ${\mc Z}(\Gamma)$ denotes the partition function of dimer configurations on with unit weights on $\Gamma$, viz. the number of perfect matchings of $\Gamma$.
In the context of the loop-erased random walk growth exponent, Kenyon \cite{Ken_Lap} showed in particular that for suitable $\Gamma_n$'s, 
$$\frac{{\mc Z}_{\rm\ dimer}(\Gamma_n\setminus\{b_0,w\})}{{\mc Z}_{\rm\ dimer}(\Gamma_n)}=O(R^{-3/4+o(1)})$$
if, say, $n$ is large enough so that $B(0,2R)\cap(\Z\times\N)\subset\Gamma_n$. See also \cite{Law_LERWexp} for an estimate without $o(1)$. 

In terms of exponents, $\frac 34=\frac 12+\frac 14$ (boundary monomer+bulk monomer exponents), in agreement with (somewhat informally)
\begin{align*}
{\rm LERW\ length}&=\frac 34=\frac 12+\frac 14={\rm boundary\ monomer}+{\rm bulk\ monomer}\\
{\rm boundary\ monomer-monomer}&=1=\frac 12+\frac 12\\
{\rm bulk\ monomer-monomer}&=\frac 12=\frac 14+\frac 14\\
{\rm bulk\ monomer}&=\frac 14=2\frac 18=2({\rm Ising\ magnetization})\\
{\rm double-dimer\ length}&=\frac 12=2\frac 14=2({\rm bulk\ monomer})
\end{align*}
based on combinatorial identities and reasonable quasi-multiplicativity assumptions.

We conclude that
$$\P^{\Gamma_n}(\diam(\ell_0)>R)=O(R^{-\frac 12+o(1)})$$
uniformly in $n$. This is somewhat crude, as we expect $O(R^{-1})$ to be the correct order of magnitude. This gives the

\begin{Lem}\label{Lem:ddtight}
Under the double-dimer measure on the half-plane $\Z\times\N$, if $\ell$ is a loop intersecting a fixed cut,
$$\P(\diam(\ell)\geq R)=O(R^{-1/2+o(1)})$$
and in particular there is a.s. no infinite path in the double-dimer configuration.
\end{Lem}

\bibliographystyle{abbrv}
\bibliography{biblio}

-----------------------

\noindent Columbia University\\
Department of Mathematics\\
2990 Broadway\\
New York, NY 10027

\end{document}